\documentclass[reqno,11pt]{amsart}
\usepackage{geometry}
\usepackage{mathtools,amssymb,amsthm,mathrsfs,color,lineno,paralist,graphicx,float}
\numberwithin{equation}{section}
\usepackage[colorlinks,
linkcolor=red,
anchorcolor=green,
citecolor=blue, 
]{hyperref}

\usepackage{enumitem}
\usepackage[T1]{fontenc}
\usepackage[utf8]{inputenc}

\usepackage{calc}
\definecolor{bleu1}{RGB}{0,57,128}
\def\bleu1{\color{bleu1}}

\usepackage{etoolbox}
\patchcmd{\section}{\normalfont}{\normalfont \bleu1}{}{}
\patchcmd{\subsection}{\normalfont}{\normalfont \bleu1}{}{}
\patchcmd{\subsubsection}{\normalfont}{\normalfont \bleu1}{}{}

\newtheorem{proposition}{Proposition}[section]
\newtheorem{theorem}{Theorem}[section]
\newtheorem{definition}{Definition}[section] 
\newtheorem{lemma}{Lemma}[section]
\newtheorem{remark}{Remark}[section]
\newtheorem{corollary}{Corollary}[section]

\newcommand{\Z}{{\mathbb Z}}
\newcommand{\C}{{\mathbb C}}
\newcommand{\R}{{\mathbb R}}
\newcommand{\Q}{{\mathbb Q}}
\newcommand{\T}{{\mathbb T}}
\newcommand{\id}{\operatorname{id}}

\usepackage{tikz,tikz-3dplot}
\tdplotsetmaincoords{80}{45}
\tdplotsetrotatedcoords{-90}{180}{-90}
\tikzset{surface/.style={draw=blue!70!black, fill=blue!40!white, fill opacity=.6}}
\newcommand{\coneback}[4][]{
	\draw[canvas is xy plane at z=#2, #1] (45-#4:#3) arc (45-#4:225+#4:#3) -- (O) --cycle;
}
\newcommand{\conefront}[4][]{
	\draw[canvas is xy plane at z=#2, #1] (45-#4:#3) arc (45-#4:-135+#4:#3) -- (O) --cycle;
}

\begin{document}

\title[]{Isospectrum of non-self-adjoint  almost-periodic Schr\"odinger operators}

\author{Xueyin Wang}
\address{
Chern Institute of Mathematics and LPMC, Nankai University, Tianjin 300071, China
}

\email{xueyinwang@mail.nankai.edu.cn}

\author {Jiangong You}
\address{
Chern Institute of Mathematics and LPMC, Nankai University, Tianjin 300071, China} \email{jyou@nankai.edu.cn}

\author{Qi Zhou}
\address{
Chern Institute of Mathematics and LPMC, Nankai University, Tianjin 300071, China
}

\email{qizhou@nankai.edu.cn}

\begin{abstract}
For non-self-adjoint almost-periodic Schr\"odinger operators, a criterion is given to guarantee that they have both the same spectrum and same  Lyapunov exponents with the discrete free Laplacian. As a byproduct, we show that the Moser-P\"oschel argument for opening gaps may not be valid for non-self-adjoint operators. 
\end{abstract}
\maketitle

\section{Introduction}
Benefiting from methods of dynamical systems and harmonic analysis,  enormous breakthroughs have been made in recent years \cite{Av0,AD,Av1,AJ,AJM,AK06,AK15,AYZ1,AYZ2,BG,BJ,GS01,GS,GSV,Hou,J,JL2,S}
in the study of {\it self-adjoint}  almost-periodic Schr\"odinger operators on $\ell^2(\Z)$ (resp. $L^2(\R)$)
\begin{equation}\label{ham}
	H_{V}= \Delta + V(\cdot),
\end{equation}
where $V(\cdot)$ are almost-periodic on $\Z$ (resp. $\R$). However, little progress has been made for {\it non-self-adjoint} almost-periodic operators (non-Hermitian quasicrystals in physical literature), and even the fundamental spectral theorem has not been established so far (may not be possible). In comparison, non-Hermitian Hamiltonians received wide attention from physicists in recent years \cite{AGU, bender, Gong, Lee, LBH, MPS, XW}, because the recent experimental advances in controlling dissipation have brought about unprecedented flexibility in engineering non-Hermitian Hamiltonians in open classical and quantum systems \cite{Gong}; non-Hermitian Hamiltonians exhibit rich phenomena without Hermitian counterparts, e.g. $\mathcal{PT}$ (parity-time) symmetry breaking, topological phase transition, non-Hermitian skin effects \cite{AGU, BBK,jiang2019interplay, LZC, Longhi}.  Of course, these observations and predictions in physical literature deserve rigorous mathematical proofs.

There are also other motivations for the mathematical study of the non-self-adjoint almost-periodic operators. Firstly, 
a striking result worth highlighting is Avila's global theory of the one-frequency quasi-periodic Schr\"odinger operators \cite{Av0}. However, if one wants to establish the quantitative global theory \cite{GJYZ}, the core is to study   
\begin{equation*}
	(H\psi)_{n}=\psi_{n+1}+\psi_{n-1}+v(x+n\alpha+ \mathrm{i}\epsilon )\psi_{n},
\end{equation*}
which is a family of non-self-adjoint operators.  Secondly, the spectrum of non-self-adjoint  Schr\"odinger  operators has deep connection with problems of the elliptic operators, such as ground states, steady states and averaging theory \cite{Kozlov1984,LWZ}. 

\subsection{Isospectrum}
In this paper, we study the spectrum of the following  non-self-adjoint almost-periodic Schr\"odinger operator  on $\ell^2(\Z)$:
\begin{equation}\label{apt}
	(H_{\lambda v,\alpha,x}\psi)_{n}=\psi_{n+1}+\psi_{n-1}+ \lambda v(x+n\alpha)\psi_{n},
\end{equation}
where $\lambda\in \mathbb{R}$ is the coupling constant, $x \in \mathbb{T}^{d}=(\mathbb{R}/2\pi\mathbb{Z})^{d}$ is called the phase with $d\in \mathbb{N}^{+}\cup\{\infty\}$,  $v:\mathbb{T}^{d}\rightarrow \mathbb{C}$ is the potential, $\alpha\in \mathbb{T}^{d}$ is the frequency satisfying that $(1,\alpha)$ is independent among $\mathbb{Q}$. In this case, the spectrum of $H_{\lambda v,\alpha,x}$  is independent of $x$ \cite{GJYZ,John}, and thus we denote it by  $\Sigma_{\lambda v,\alpha}$. 

$\mathcal{PT}$ symmetric operators constitute an important  class of non-self-adjoint operators coming from  quantum mechanics \cite{KYZ}. Recall that \eqref{ham} is $\mathcal{PT}$ symmetric, if $\overline{V}(n)=V(-n)$ \cite{bender}. In the almost-periodic setting where the potential $v:\mathbb{T}^{d}\rightarrow \mathbb{C}$, one can extend the definition to $\overline{v}(x)=v(-x)$, since the spectrum  $\Sigma_{v,\alpha}$  is independent of $x$.

It was first observed by Bender and Boettcher \cite{bender} that a large class of $\mathcal{PT}$ symmetric operators have real spectrum. This observation has a profound significance in that it not only suggests a possibility of $\mathcal{PT}$ symmetric modification of the conventional quantum mechanics that considers observables as self-adjoint operators \cite{KYZ}, but also goes far beyond quantum mechanics and has spread to many branches of physics \cite{RDM}. Thus, a basic mathematical question is to ask which class of  $\mathcal{PT}$ symmetric operators have real spectrum.

As a warming up example, let us first look at the heuristic example
\begin{equation}\label{s1}
	(H_{\lambda\exp,\alpha}\psi)_{n}=\psi_{n+1}+\psi_{n-1}+\lambda{\rm e}^{{\rm i}(x+n\alpha)} \psi_{n},
\end{equation}
proposed by Sarnak \cite{Sarnak}, whose spectrum has already been completely known.

\begin{theorem}[\cite{boca,BF,Sarnak}] \label{sarexample}
	For any $\alpha\in \mathbb{R}\backslash\mathbb{Q}$, we have the following:
	\begin{enumerate}[font=\normalfont, label={(\arabic*)}]
		\item \label{item:s1}If $|\lambda|\leq 1$,  then the spectrum of \eqref{s1} is a real interval: 
		$$\Sigma_{\lambda \exp,\alpha}=[-2,2].$$
		\item \label{item:s2}If $|\lambda| > 1$, denote $\xi=\ln |\lambda|$, then the spectrum  of \eqref{s1} is an  ellipse given by
		\begin{equation*}
			\Sigma_{\lambda \exp,\alpha}=\bigg\{ E\in \C: \ \	\Big(\frac{\mathrm{Re}E}{\cosh \xi} \Big)^{2} + \Big(\frac{\mathrm{Im}E}{\sinh \xi} \Big)^{2} =4\bigg\}.
		\end{equation*}
	\end{enumerate} 
\end{theorem}

Theorem \ref{sarexample} was proved by Sarnak \cite{Sarnak} in the case $|\lambda|\neq 1$ and $\alpha$ is Diophantine.  It was generalized to all $\lambda\in \mathbb{R}$ and $\alpha\in \mathbb{R}\backslash\mathbb{Q}$ by Boca \cite{boca} and Borisov-Fedotov \cite{BF} (the proof didn't appear yet) independently. In this paper,  we will give a simple proof of Theorem \ref{sarexample} by Avila's global theory of one-frequency analytic $\mathrm{SL}(2,\mathbb{C})$ cocycles \cite{Av0}.

The phenomenon of being isospectral to the free Laplacian $H_{0}$, described by Theorem \ref{sarexample}\ref{item:s1},  is of particular interest. It has roots in the study of the non-self-adjoint differential operator with periodic potential, where the study is relatively complete now \cite{K1, T,V}.
The famous Borg's uniqueness theorem \cite{Bo}  for the Hill operator
\begin{equation*}
	(H_{v}\psi)(t)=-\psi''(t)+ v(t)\psi(t),
\end{equation*}
states that if $v\in L_{loc}^2(\R)$  is real-valued, then  $\Sigma_{v}=[0,\infty )$ if and only if  $v\equiv 0$ a.e.. However, in the case of a complex-valued periodic potential $v$, the situation is very different. As it was proved by Gasymov \cite{Gas} (see also \cite{GU})  that if the Hill operator satisfies
\begin{equation}\label{gas}
	v(t)=\sum_{\mathbf{k}=1}^{\infty} \widehat{v}_{\mathbf{k}}{\rm e}^{{\rm i}\mathbf{k}t} \quad  \text{with}  \quad \sum_{\mathbf{k}=1}^{\infty} |\widehat{v}_{\mathbf{k}}|<\infty,
\end{equation} 
then $\Sigma_{v}=[0,\infty )$ or say $H_{v}$ is isospectral to $H_{0}$. However, it is still open whether, under some smoothness requirements, any operator with periodic potential $v(x)$ isospectral to $H_{0}$ must be a ``Gasymov potential'', i.e. of the form given by \eqref{gas}, or the complex conjugate of a Gasymov potential \cite{P}.

In the non-periodic setting,  Killip-Simon \cite{KS} largely extended Borg's uniqueness theorem \cite{Bo}, and proved that for self-adjoint discrete Schr\"odinger operator \eqref{ham} with 
$V:\mathbb{Z}\rightarrow \mathbb{R}$, $H_{V}$ is isospectral to $H_{0}$ if and only if  $V\equiv 0$. Theorem \ref{sarexample} shows that being isospectral to $H_{0}$ does not imply $v\equiv 0$ for complex  quasi-periodic potential \eqref{apt}.

The main ambition of this paper is to explore the structure of complex potential (not necessary to be $\mathcal{PT}$ symmetric), and to give a criterion to ensure the corresponding operators  \eqref{apt} are  isospectral to $H_{0}$. Before stating the results, we first introduce some notations. For any ${\bf k}\in \mathbb{Z}^{d}$, we define $|{\bf k}|_{\eta}  = \sum_{j\in\mathbb{N}} \langle j\rangle^{\eta} |{\bf k}_{j}|$, where $\langle j\rangle:=\max\{1,j\}$ and $\eta>0$ is a fixed constant. Let $\mathbb{Z}^{d}_{*}$ be the set of integer vectors with finite support $
\mathbb{Z}^{d}_{*} = \{{\bf k}: 0<|{\bf k}|_{\eta}  <\infty\}$. Clearly if $d\in \mathbb{N}^{+}$, then $\mathbb{Z}^{d}_{*}=\mathbb{Z}^{d}\backslash\{0\}$. Let $\mathbb{T}^{d}_{h}$ be the complexified torus defined by
\begin{equation*}
	\mathbb{T}_{h}^{d} := \Big\{x\in \mathbb{C}^{d}:{\rm Re} x_{j}\in \mathbb{T}, |{\rm Im} x_{j}|<h \langle j\rangle^{\eta} \Big\},
\end{equation*}
and denote by $C^{\omega}(\mathbb{T}^{d}_{h},\mathbb{C})$ the space  of bounded analytic complex-valued functions equipped with norm $\|v\|_{h} = \sum_{{\bf k}} |\widehat{v}_{\bf k}|{\rm e }^{h |{\bf k}|_{\eta}}$. 

Let $d\in \mathbb{N}^{+} \cup \{\infty\}$.  We assume that the frequency $\alpha=(\alpha_j)$ belongs to  the $d$ dimensional cube $\mathcal{R}_0:= [1,2]^{d}$, which is endowed with the probability measure $\mathcal{P}$
induced by the product measure of the $d$ dimensional cube $\mathcal{R}_0$.
The following almost-periodic Diophantine frequencies  were first defined by Bourgain \cite{Bourgain2005}:
\begin{definition}[\cite{Bourgain2005}]
	Given $\gamma\in (0,1),\tau>1$, we denote by ${\rm DC}_{\gamma,\tau}^{d}$ the set of Diophantine frequencies
	\begin{equation}\label{Almost}
		\inf\limits_{n\in\Z}|\langle \mathbf{k},\alpha\rangle-n|\geq \gamma\prod\limits_{j\in\mathbb N}\frac{1}{1+\langle j\rangle^\tau|k_j|^\tau},\quad \forall \ \mathbf{k} \in\mathbb{Z}^{d}_{*}, 
	\end{equation}
	and denote ${\rm DC}^{d} = \cup_{\gamma>0}{\rm DC}_{\gamma,\tau}^{d}$.
\end{definition}
As proved in \cite{Bourgain2005,Biasco2019}, for any $\tau>1$, Diophantine frequencies  ${\rm DC}_{\gamma,\tau}^{d}$ are typical in the set $\mathcal{R}_0$ in the sense that there exists a positive constant $C(\tau)$ such that
\begin{equation*}
	\mathcal{P}( \mathcal{R}_0 \backslash {\rm DC}_{\gamma,\tau}^{d}) \leq C(\tau) \gamma.
\end{equation*} 

We also denote 
\begin{equation}\label{cone-r}
	\Gamma_{r} = \mathbb{Z}^{d}_{*} \cap \bigg\{\mathbf{k}: \sum_{j} \langle j\rangle^{\eta}{\bf k}_{j} w_{j} \geq  r|{\bf k}|_{\eta}, \quad \text{with}\quad \sum_{j}w_{j}=1, w_j>0  \bigg\},
\end{equation}
where $\Gamma_{r}$  is an integer cone whose angle is less than $\pi$ strictly, as shown in  Figure \ref{acone}. Once we have these, now we are ready to  state our main theorem.

\begin{figure}[htbp]
	\centering
	\begin{tikzpicture}[tdplot_main_coords]	
		\coordinate (O) at (0,0,0);
		\draw (0,0,-0.5) -- (O);
		\foreach \x in {-4,-3,...,4}
		\foreach \y in {-4,-3,...,4}
		{
			\draw[gray,very thin]  (\x,-4) -- (\x,4);
			\draw[gray,very thin]  (-4,\y) -- (4,\y);
		}
		\draw[->] (-6,0,0) -- (6,0,0) node[right]{};
		\draw[->] (0,6,0) -- (0,-6,0) node[right]{};
		\coneback[surface]{2}{3}{25}
		\draw[->] (O) -- (0,0,3) node[above]{};
		\conefront[surface]{2}{3}{25}
		\draw[->] (2.5,2.5,2)node[right]{$\Gamma_{r}$} -- (2.2,2.2,2)node{};
	\end{tikzpicture}
	\caption{Integer cone}
	\label{acone}
\end{figure}
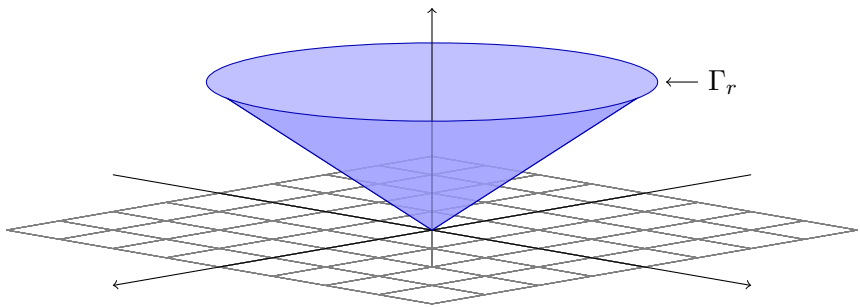

\begin{theorem}\label{maintheorem1}
	Let $d\in \mathbb{N}^{+} \cup \{\infty\}$, $\eta>0$, $h>0$,  $r\in (0,1]$,  $\alpha\in {\rm DC}^{d}$. Suppose that $$v(x)= \sum_{{\bf k} \in \Gamma_{r}}\widehat{v}_{\bf k}\mathrm{e}^{{\rm i} \langle \mathbf{k},x\rangle} \in C^{\omega}(\mathbb{T}^{d}_{h},\mathbb{C}).$$
	Then there exists $\lambda_0=\lambda_{0}(\eta,h,r,\alpha,\|v\|_{h})$ such that  $\Sigma_{\lambda v,\alpha}=[-2,2]$ if $ |\lambda|<\lambda_0$. 
\end{theorem}

\begin{remark}
The smallness condition of the coupling constant $|\lambda|$ is necessary due to Theorem \ref{sarexample}.

\end{remark}

\begin{remark}
	If  $d<\infty$, the assumption \eqref{Almost} can be replaced by the  standard  Diophantine condition
	\begin{equation*}
		{\rm DC}^d_{\gamma,\tau'}:=\Big\{\alpha \in\R^d:  \inf_{n \in \Z}\left| \langle \mathbf{k},\alpha \rangle - n \right|
		> \frac{\gamma}{|\mathbf{k}|^{\tau'}},\quad \forall\  \mathbf{k}\in\Z^d\backslash\{0\} \Big\}.
	\end{equation*}	
\end{remark}

\begin{remark}
If $\widehat{v}_{\bf k}\in \R$ for any $\mathbf{k}\in \Gamma_{r}$, then the potential $v$ is $\mathcal{PT}$ symmetric. However, the key assumption for us is the cone structure $\Gamma_{r}$, whether $\widehat{v}_{\bf k}$ is real or not is not important.
\end{remark}

Note that Sarnak \cite{Sarnak} also extended his result to multi-frequency case: for any  Diophantine frequency $\alpha$,  he constructed one $\mathcal{PT}$ symmetric $v$, and showed that $H_{\lambda v,\alpha,x}$ is isospectral to the discrete free Laplacian if $|\lambda|$ is small enough\footnote{While Sarnak stated the result in the continuous setting, the method can clearly be carried out in the discrete case, as point out by him at the end of Section 2 \cite{Sarnak}.}. Our result not only generalizes Sarnak's result \cite{Sarnak} to the almost-periodic case, but also (more importantly) finds that the cone structure $\Gamma_{r}$ in \eqref{cone-r} for the Fourier coefficients of  $v$ is a sufficient (almost optimal) condition to ensure that $H_{\lambda v,\alpha,x}$ is isospectral to the discrete free Laplacian.

To understand Theorem \ref{maintheorem1} more clearly, we  look at the case $d=1$, where $\Gamma_{r}=\mathbb{Z}^+$, consequently we have the following:
\begin{corollary}
	Suppose that $h>0$, $\alpha\in {\rm DC}^{1}$, $v(x)= \sum_{\mathbf{k}>0} \widehat{v}_{ \mathbf{k}}\mathrm{e}^{{\rm i}  \mathbf{k}x}.$  If
	$|\lambda|<\lambda_1(h,\alpha,\|v\|_{h})$ which is small enough, then $\Sigma_{\lambda v,\alpha}=[-2,2]$.
\end{corollary}

We remark that the phenomenon described in Theorem \ref{maintheorem1} is totally different from the self-adjoint case where having open gaps is a typical phenomenon \cite{ABD, Eli, GS, GSV, Puig}. The most important  example is the  almost Mathieu operator:
\begin{equation*}
	(H_{2\lambda\cos,\alpha,x}\psi)_{n}=\psi_{n+1}+\psi_{n-1}+2\lambda\cos(x+n\alpha)\psi_{n},
\end{equation*}
whose spectrum is a Cantor set for all $\lambda\neq 0$, $\alpha\in \mathbb{R}\backslash \mathbb{Q}$ and all $x\in\mathbb{T}$ \cite{Av1}.
We also remark that  the cone structure assumption \eqref{cone-r} is  necessary due to the following counter-example where the angle of the cone is $\pi$.

\begin{proposition}\label{maintheorem2}
	Let $\alpha\in \mathbb{R} \backslash \mathbb{Q}$,  $|\lambda|\in (0,1)$, $|\epsilon|< -\log |\lambda|$, and 
	\begin{equation*}
		v_{\epsilon}(x)=2\lambda \cos(x+\mathrm{i}\epsilon).
	\end{equation*}
	Then $\Sigma_{v_{\epsilon},\alpha} = \Sigma_{2\lambda \cos,\alpha}$  is a Cantor set.
\end{proposition}

\subsection{Lyapunov exponent}
In addition, the Lyapunov exponents of $H_{\lambda v,\alpha,x}$ in Theorem \ref{maintheorem1} can be exactly calculated. Recall that the eigenvalue equations $H_{\lambda v,\alpha,x}\psi=E\psi$ are equivalent to a certain family of the discrete dynamical systems, so called Schr\"odinger cocycle $(\alpha,S_{E,\lambda v})\in \mathbb{T}^{d}\times{\rm SL}(2,\mathbb{C})$, i.e., 
\begin{equation*}
	\begin{pmatrix}
		\psi_{n+1}\\ \psi_{n}
	\end{pmatrix}=S_{E,\lambda v}(x+n\alpha )\begin{pmatrix}
		\psi_{n}\\ \psi_{n-1}
	\end{pmatrix}, \ \text{where} \ S_{E,\lambda v}(x)=\begin{pmatrix}
		E-\lambda v(x)&-1\\1&0
	\end{pmatrix}.
\end{equation*}
Any formal solution $(\psi_{n})_{n\in \mathbb{Z}}$  can be reconstructed from the transfer matrix $S_{n}$, which is defined by $S_{0}=\mathrm{id}$, and for $n\geq 1$, by
\begin{equation*}
	S_{n}(x)=S_{E,\lambda v}(x+(n-1)\alpha) \cdots S_{E,\lambda v}(x), \quad S_{-n}(x)= S_{n}(x-n\alpha)^{-1}.
\end{equation*}
The Lyapunov exponent of  $(\alpha,S_{E,\lambda v})$,   denoted by $L(E)$, is defined   by
\begin{equation*}
	L(E)=\lim_{n\rightarrow \infty} \frac{1}{n} \int_{\mathbb{T}^{d}} \log \|S_{n}(x)\| \mathrm{d}x.
\end{equation*}

In general, it is hard to give a precise expression of the Lyapunov exponent $L(E)$ except for some very special examples.  It is  well-known  that $L(E)=\max\{0,\log |\lambda|\}$  in the spectrum \cite{Av0,BJ} for the almost Mathieu operator, however, the formula of $L(E)$  for $E$ outside the spectrum is  not known. For other analytic quasi-periodic operators, it is almost impossible to have a precise expression of $L(E)$ even though $E$ is in the spectrum  \cite{Av0}. Up to now, the only quasi-periodic Schr\"odinger operator whose Lyapunov exponent can be calculated explicitly for all $E\in \mathbb{C}$ is the Maryland model, but this is due to the unboundedness and monotonicity of the potential \cite{GFP, JL}.

In the following, for any  $E\in \mathbb{C}$, we give the exact expression of $L(E)$ for the operators defined in Theorem \ref{sarexample} and Theorem \ref{maintheorem1}.   For Sarnak's  example \eqref{s1}, we have 
\begin{theorem}\label{maintheorem3-1}
	Let $\alpha\in \mathbb{R}\backslash\mathbb{Q}$ and $v(x)=\lambda\mathrm{e}^{\mathrm{i}x}$. Then for any $\lambda \in \mathbb{R}\backslash \{0\}$, its Lyapunov exponent satisfies	
	\begin{equation}\label{sarlya1}
		L(E)=\max\{0,\log |\lambda|\}, \quad \forall \ E\in \Sigma_{\lambda \exp,\alpha}.
	\end{equation}	
	Moreover, we have the following:	
	\begin{equation}\label{sarlya2}
		L(E)=\max\bigg\{\log \Big|\frac{E}{2}+\frac{\sqrt{E^{2}-4}}{2}\Big|,\log |\lambda|\bigg\}, \quad \forall \ E\in \mathbb{C}.
	\end{equation}
\end{theorem}
\begin{remark}
	 \eqref{sarlya1} was also announced by Borisov-Fedotov \cite{BF}, while  \eqref{sarlya2}  is totally new.   Our proofs are new and based on  Avila's global theory \cite{Av0}.
\end{remark}

If the potential has cone structure $\Gamma_r$, Theorem \ref{maintheorem1} states that the corresponding  Schr\"odinger operator is isospectral to the discrete free Laplacian.  The following theorem shows that they also share the same  Lyapunov exponent with the discrete free Laplacian.

\begin{theorem}\label{maintheorem3}
	Under the same assumptions as in Theorem \ref{maintheorem1}, we have 	
	\begin{equation*}
		L(E)=\log \Big|\frac{E}{2}+\frac{\sqrt{E^{2}-4}}{2}\Big|, \quad \forall \ E \in \mathbb{C}.
	\end{equation*}
	In particular, 
	\begin{equation*}
		L(E)=0, \quad \forall \ E\in \Sigma_{\lambda v,\alpha}.
	\end{equation*}	
\end{theorem}

\subsection{Failure of Moser-P\"{o}schel argument}

Almost reducibility is an effective approach to deal with the spectral problems of almost-periodic operators, especially for the small potentials \cite{AJ, Cai2, ds,Eli,LYZZ,MP}.  Recall that two cocycles $(\alpha, A), (\alpha, A') \in \mathbb{T}^{d} \times C^{\omega}(\mathbb{T}^{d}, \mathrm{SL}(2,\mathbb{C}))$ are analytically conjugated if there exists $B\in C^{\omega}(2\mathbb{T}^{d}, \mathrm{SL}(2,\mathbb{C}))$ such that
\begin{equation*}
	B(x+\alpha)^{-1} A(x)B(x)=A'(x).
\end{equation*}
The almost-periodic cocycle $(\alpha, A)$ is almost reducible if the closure of its analytic conjugates contains a constant matrix. Moreover, the cocycle is reducible if it is analytically conjugated to a constant matrix. In the self-adjoint setting, if the potential is small, and the frequency $\alpha \in \mathrm{DC}^d$ with $d\in \mathbb{N}^+$,  Eliasson's famous result \cite{Eli} states that for all $E\in \Sigma_{\lambda v,\alpha}$ the cocycle $(\alpha,S_{E,\lambda v})$ is almost reducible. If, furthermore, the rotation number is Diophantine with respect to $\alpha$ or rational dependent, then  $(\alpha, S_{E,\lambda v})$ is in fact reducible.  For the non-self-adjoint operator, the rotation number is not well-defined since the projection of $\mathrm{SL}(2,\mathbb{C})$ cocycle is not a circle diffeomorphism. However, we can find another object that plays the same role as the rotation number. Indeed, for any $E\in [-2,2]$, if we define 
$$\rho=\rho(E):= \arccos \Big(\frac{E}{2}\Big)\mod \pi,$$ then we have the following:

\begin{theorem}\label{maintheorem4}
	Under the same assumptions as in Theorem \ref{maintheorem1},   we have 
	\begin{enumerate}[font=\normalfont, label={(\arabic*)}]
		\item If $\rho\in {\rm DC}(\alpha)=\cup_{\kappa>0} {\rm DC}_{\kappa,\tau}(\alpha)$, where 
		\begin{equation*}
			{\rm DC}_{\kappa,\tau}(\alpha) :=\bigg\{\rho\in \mathbb{R}: \| 2\rho -\langle {\bf k},  \alpha\rangle \|_{\mathbb{T}}> \kappa\prod\limits_{j\in\mathbb N}\frac{1}{1+\langle j\rangle^\tau|k_j|^\tau}, \ \forall \ {\bf k}\in \mathbb{Z}^{d}_{*}\bigg\},
		\end{equation*}
		then  
		$(\alpha, S_{E,\lambda v})$ is reducible to $\big(\alpha,\begin{pmatrix}
			\mathrm{e}^{\mathrm{i}\rho}&0\\
			0&\mathrm{e}^{-\mathrm{i}\rho}
		\end{pmatrix}\big)$.
		\item If $2\rho=\langle {\bf k},\alpha\rangle \mod 2\pi$ for some $\mathbf{k}\in \mathbb{Z}^{d}_{*}$, then  $(\alpha, S_{E,\lambda v})$ is reducible to $(\alpha,A)$ where $A=\begin{pmatrix}
			1&0\\0&1
		\end{pmatrix}$ or $A=\begin{pmatrix}
			1&\zeta\\0&1
		\end{pmatrix}$ with $\zeta\neq0$.
	\end{enumerate}	
\end{theorem}

In the self-adjoint case, by the well-known Moser-P\"{o}schel argument \cite{MP}, it is known that $(\alpha, S_{E,\lambda v})$ is reducible to identity if and only if $E$ is located at the edges of the collapsed gaps. However, in the non-self-adjoint case, one may need more caution due to the following result. 

\begin{theorem}\label{maintheorem5}
	Let $d\in \mathbb{N}^{+} \cup \{\infty\}$,  $\alpha\in {\rm DC}^{d}$, and $v(x)=\mathrm{e}^{\mathrm{i} \langle \mathbf{m},x\rangle}$. Then there exists $E\in (-2,2)$ such that $(\alpha, S_{E,\lambda v})$ is reducible to $\begin{pmatrix}
		1&\zeta\\ 0 &1
	\end{pmatrix}$ with $\zeta\neq 0$,
	provided that $|\lambda|$ is sufficiently small.
\end{theorem}
By Theorem \ref{maintheorem1} and Theorem \ref{maintheorem5}, we see that even though all gaps are collapsed, the cocycle $(\alpha, S_{E,\lambda v})$ may still not be reducible to identity. Therefore, Moser-P\"{o}schel argument \cite{MP} does not work for non-self-adjoint almost-periodic operators.

\subsection{Methods and mechanism}  Although Theorem \ref{maintheorem1} is an extension of Sarnak's result \cite{Sarnak}, our method is completely different. In fact, the operator \eqref{s1} is very special, as taking Fourier transforms in trying to  solve $(H_{\lambda\exp,\alpha}-E)\psi=0$, one finds 
\begin{equation*}
	\lambda\widehat{\psi}(p+\alpha) = (E-2\cos p) \widehat{\psi}(p),
\end{equation*}
which is easily iterated. Sarnak \cite{Sarnak} studied the behavior of $\prod_{p=0}^{n} (E-2\cos (p+ \alpha))$ by combining Birkhoff ergodic theorem and nature of $\alpha\beta$-sets studied initially by Engelking \cite{Enge} and Katznelson \cite{Katz}. As we can see,  the method of \cite{Sarnak} depends on the duality transformation, clearly fails if $\alpha\in \T^d$ with $d=\infty$, i.e. the true almost-periodic case.

We know that the spectrum of self-adjoint operators always stays in the real axis, and the non-self-adjointness would push the spectrum out. In this paper, we give a criterion for non-self-adjoint almost-periodic Schr\"odinger operators \eqref{apt} to have
real interval spectrum.  More importantly,  one can see the mechanism of the spectrum being real and staying an interval from our proof.
Let us explain the main ideas. Our approach is based on the quantitative almost reducibility of the Schr\"odinger cocycle.  In the self-adjoint case \cite{Eli,LYZZ}, the potential $v$ is real ($\widehat{v}_{-\bf k}= \overline{\widehat{v}_{\bf k}}$) which implies that {\it double resonances\footnote{Second Melnikov condition in Hamiltonian systems.}}
\begin{equation*}
	\|\langle {\bf k},\alpha \rangle\pm 2\rho\|_{\mathbb{T}} \sim 0
\end{equation*}
must occur, which causes the uniform hyperbolicity of the Schr\"odinger cocycle and makes the corresponding gap open. In the non-self-adjoint case, the potential is not  real anymore, which gives us a chance to avoid the {\it double resonances}. For the potential $v$ defined in Theorem \ref{maintheorem1}, we have   $\widehat{v}_{\bf k}\cdot \widehat{v}_{-\bf k}=0$.
During the KAM iteration steps, we will prove that at the cost of shrinking $r$, the cone structure $\Gamma_{r}$ is preserved 
and there exists only {\it single resonance} (${\bf k}$ or ${-\bf k}$) in each iteration step. Thus, the interval spectrum may survive.  

More precisely, we can prove that the Schr\"odinger cocycle is reducible to $(\alpha, A_{n}\mathrm{e}^{F_{n}})$ where $A_{n}\in \mathrm{SL}(2,\mathbb{C})$ with eigenvalues  $\mathrm{e}^{\pm\mathrm{i}\xi_{n}}$ and $F_{n}$ goes to zero.   The structure of $\Gamma_{r}$ guarantees that the average of the perturbation $F_{n}$ is always zero, and thus $\mathrm{Im} \xi_{n}$ is fixed during the KAM iteration, this is the reason why the Schr\"odinger operator \eqref{apt} has real spectrum. In addition, we will prove that the Schr\"odinger cocycle $(\alpha, S_{E,\lambda v})$ is always almost reducible to the Laplace cocycle $(\alpha, S_{E,0})$, which implies that the Schr\"odinger operator shares both the  spectrum and the Lyapunov exponent with the discrete free Laplacian.

\subsection{Organization of the paper}
The rest of this paper is organized in the following way. Some basic definitions are given in Section 2. In Section 3, we study the one step of KAM iteration for  ${\rm SL}(2,\mathbb{C})$-valued cocycle with integer cone condition. In Section 4, we obtain the reducibility of $\mathrm{SL}(2,\mathbb{C})$-valued cocycle. In Section 5, as applications, we prove Theorem \ref{maintheorem1},  Theorem \ref{maintheorem3},  Theorem \ref{maintheorem4} and Theorem \ref{maintheorem5}. Finally, we prove Theorem \ref{sarexample}, Theorem \ref{maintheorem3-1} and Proposition \ref{maintheorem2} in Section 6. In Appendix, we give the proof of Lemma \ref{hy}.

\section{Preliminary}

\subsection{Almost-periodic cocycle, Lyapunov exponent} \label{co-ly}
Let $\Omega$ be a compact metric space and $(\Omega, \nu, T)$ be ergodic. A cocycle $(\alpha, A)\in \R\backslash
\Q\times C^\omega(\Omega, {\rm SL}(2,\C))$ is a linear skew-product:
\begin{align*}
(T,A):&\ \ \Omega \times \C^2 \to \Omega \times \C^2,\\
 &(x,\phi) \mapsto (Tx,A(x)  \phi).
\end{align*}
For $n\in\mathbb{Z}$, $A_n$ is defined by $(T,A)^n=(T^n,A_n)$. Thus, $A_{0}(x)=\mathrm{id}$,
\begin{equation*}
A_{n}(x)=\prod_{j=n-1}^{0}A(T^{j}x)=A(T^{n-1}x)\cdots A(Tx)A(x),\quad \forall \ n\geq 1,
\end{equation*}
and $A_{-n}(x)=A_{n}(T^{-n}x)^{-1}$. The Lyapunov exponent is defined as
\begin{equation*}\quad
L(T,A)=\lim_{n\rightarrow\infty}\frac{1}{n}\int_{\Omega}\log\|A_{n}(x)\|\mathrm{d}\nu (x).
\end{equation*}
We are mainly interested in the case   $\Omega=\mathbb{T}^d$, $\mathrm{d}\nu=\mathrm{d}x$ is Lebesgue measure, and $T=R_\alpha$, with $(1,\alpha)$ rational independent. If $d\in \mathbb{N}^{+}$, then $(\alpha,A)=: (R_{\alpha},A)$ defines a  quasi-periodic cocycle. If $d=\infty$, then $(\alpha,A)$ defines an almost-periodic cocycle.

We say $(\alpha, A)$ is uniformly hyperbolic if there exist two continuous functions $u,s: \mathbb{T}^{d} \rightarrow \mathbb{PC}^{2}$, called the unstable and stable directions such that for any $n\geq 0$,
\begin{align*}
	&\|A_{n}(x) \phi\| \leq C {\rm e}^{-c n} \|\phi\|, \quad \forall \ \phi\in s(x),\\
	&\|A_{-n}(x) \phi\| \leq C {\rm e}^{-c n} \|\phi\|, \quad \forall \ \phi\in u(x),
\end{align*}
for some constants $C,c>0$. Moreover, $u(\cdot), s(\cdot)$ are invariant under the dynamics:
\begin{equation*}
	A(x) \cdot u(x) =u(x+\alpha), \quad A(x) \cdot s(x) =s(x+\alpha),
\end{equation*}
where $A \cdot x$ denoted the ${\rm SL}(2,\mathbb{C})$ action on the projective space $\mathbb{PC}^{2}$. If $(\alpha,A)$ is uniformly hyperbolic, then $L(A)>0$. From now on, $(\alpha, A)\in \mathcal{UH}$ means $(\alpha,A)$ is uniformly hyperbolic.

\subsection{Schr\"odinger operators and  Schr\"odinger cocycles}

Let $\Omega$ be a compact metric space, $T:\Omega\rightarrow \Omega$ a  homeomorphism, and $v:\Omega\rightarrow \C$ a complex-valued continuous function. We consider the following complex-valued dynamical defined Schr\"odinger operators:
\begin{equation*}
(H_{x}\psi)_n=\psi_{n+1}+\psi_{n-1}+v(T^nx)\psi_n,\quad n\in\Z,
\end{equation*}
and denote $\Sigma_x$ by the spectrum of $H_{x}$. Note that any formal solution $\psi=(\psi_n)_{n \in \Z}$ of $H_{x}\psi=E \psi$ satisfies
\begin{equation*}
	\begin{pmatrix}
		\psi_{n+1}\\
		\psi_n
	\end{pmatrix}
	= S_{E,v}(T^n x) \begin{pmatrix}
		\psi_{n}\\
		\psi_{n-1}
	\end{pmatrix},\quad n \in \Z,
\end{equation*}
where
\begin{equation*}
	S_{E,v}(x):=
	\begin{pmatrix}
		E-v(x) & -1\\
		1 & 0
	\end{pmatrix},   \quad E\in\C.
\end{equation*}
 We call $(T,S_{E,v})$  \textit{Schr\"odinger cocycles}.  The spectrum $\Sigma_{x}$ is closely related with the dynamical behavior of the Schr\"odinger cocycle  $(T,S_{E,v})$.  In the self-adjoint case, i.e. the potential $v$ is real-valued, then by the well-known result of Johnson \cite{John},  $E\notin \Sigma_{x}$ if and only if $(T,S_{E,v})\in\mathcal{UH}$.
The following result extends  Johnson's result \cite{John} to the non-self-adjoint case.

\begin{proposition}\label{equi} \cite{GJYZ}
	Suppose that   $v:\Omega\rightarrow \C$ a complex-valued continuous function, $(\Omega,T)$ is  minimal. Then there is some $\Sigma\subset \C$ such that $\Sigma_x=\Sigma$ for all $x\in\Omega$. Moreover,  $E\notin \Sigma$ if and only if $(T,S_{E,v})\in \mathcal{UH}$.
\end{proposition}

\subsection{Global theory of one-frequency quasi-periodic cocycles.}\label{acceleration}
We give a short review of  Avila's global theory of one-frequency quasi-periodic $\mathrm{SL}(2,\mathbb{C})$ cocycles \cite{Av0}.  Let $\alpha\in \mathbb{R}\backslash\mathbb{Q}$, suppose that $A\in C^\omega(\T, \mathrm{SL}(2,\mathbb{C}))$ admits a holomorphic extension to $\{|\mathrm{Im}  x |<h\}$. Then for
$|\epsilon|<h$, we define $A_\epsilon \in
C^\omega(\T, \mathrm{SL}(2,\mathbb{C}))$ by $A_\epsilon(\cdot)=A(\cdot+\mathrm{i} \epsilon)$, and define the
the acceleration of $(\alpha, A_{\epsilon})$ as follows:
\begin{equation*}
\omega(\alpha, A_{\epsilon})=\lim_{h\to 0^{+}}\frac{L(\alpha, A_{\epsilon+h})-L(\alpha,A_{\epsilon})}{h}.
\end{equation*}

It follows from the convexity and continuity of the Lyapunov exponent that the acceleration is an upper semi-continuous function of parameter $\epsilon$. The key property of the acceleration is that it is quantized: 

\begin{theorem}[Quantization of acceleration \cite{Av0}]\label{acce}
Suppose that  $(\alpha,A)\in \mathbb{R}\backslash\mathbb{Q} \times  C^{\omega}(\mathbb{T},$ $ \mathrm{SL}(2,\C))$, then $\omega(\alpha,A_{\epsilon})\in\mathbb{Z}$. 
\end{theorem}

For uniformly hyperbolic cocycles, Avila \cite{Av0} proved the following equivalent characterization:

\begin{proposition}\label{uh}\cite{Av0}
	Let $(\alpha,A)\in \mathbb{R} \backslash \mathbb{Q} \times C^{\omega}(\mathbb{T}, \mathrm{SL}(2,\mathbb{C}))$. Assume that $L(\alpha, A)>0$. Then $(\alpha, A)\in \mathcal{UH}$ if and only if $L(\alpha,A(\cdot+\mathrm{i}\epsilon))$ is affine with respect to $\epsilon$ around $\epsilon=0$.
\end{proposition}

\section{Quantitative almost reducibility}

As mentioned in the Introduction, our approach is based on quantitative almost reducibility. The philosophy is that nice quantitative almost reducibility  brings the precise estimates of the growth on the Schr\"odinger cocycle.

\subsection{Auxiliary Banach space}
We first introduce the auxiliary Banach space related to  the integer cone $\Gamma_{r}$.
Recall that  the integer cone $\Gamma_{r}$ is defined as
\begin{equation*}
	\Gamma_{r} = \mathbb{Z}^{d}_{*} \cap\{\mathbf{k}:[\![\mathbf{k}]\!]\geq r|{\bf k}|_{\eta}\}, 
\end{equation*}
where $[\![\mathbf{k}]\!]=\sum_{j} \langle j\rangle^{\eta}{\bf k}_{j} w_{j}$ with $\sum_{j}w_{j}=1, w_j>0$. 
For a given integer cone $\Gamma_{r}$, we define the space 
\begin{equation*}
	\mathcal{B}_{h,r} [*]= \{F\in C^{\omega}(\mathbb{T}^{d}_{h}, *): \widehat{F}_{\bf k} =0, \ \ \forall \ {\bf k}\in \mathbb{Z}^{d}\backslash \Gamma_{r}\},
\end{equation*}
where $*$ could be $\mathbb{C}$ or ${\rm sl}(2,\mathbb{C})$, and we abbreviate $	\mathcal{B}_{h,r} [{\rm sl}(2,\mathbb{C})]$ by $\mathcal{B}_{h,r}$ without ambiguity. Since ${\bf 0}\notin \Gamma_{r}$, it holds that $\widehat{F}_{\bf 0} =0$ for any $F\in \mathcal{B}_{h,r}[*]$.

For any set $W\subset \Gamma_{r}$ and $N>0$, we define the truncated set  and residual set of $W$ as
\begin{equation*}
	\mathcal{T}_{N}W=\{{\bf k}\in W: |{\bf k}|_{\eta}\leq N\},\quad
	\mathcal{R}_{N}W=\{{\bf k}\in W: |{\bf k}|_{\eta}> N\}.		
\end{equation*} 
And we also define the truncated operator $\mathcal{T}_{N}$ and residual operator $\mathcal{R}_{N}$ by 
\begin{equation*}
	(\mathcal{T}_{N} F)(x)= \sum_{{\bf k}\in \mathcal{T}_{N}\Gamma_{r}} \widehat{F}_{\bf k} {\rm e}^{{\rm i}\langle {\bf k},x \rangle}, \quad (\mathcal{R}_{N} F)(x)= \sum_{{\bf k}\in \mathcal{R}_{N}\Gamma_{r}} \widehat{F}_{\bf k} {\rm e}^{{\rm i} \langle {\bf k},x \rangle}.
\end{equation*}
The following are some  basic properties of the space $\mathcal{B}_{h,r} [*]$.

\begin{proposition} \label{banach} 
For any $0<r<r' \leq 1 $, $h>0$, we have the following:
\begin{enumerate}[font=\normalfont, label={(\arabic*)}]
\item \label{item:subset}$(\mathcal{B}_{h,r}[*], \|\cdot\|_{h})$ is a Banach space with $\mathcal{B}_{h,r'} \subset \mathcal{B}_{h,r}$.


\item \label{item:closed}$\|[F,G]\|_{h} \leq 2 \|F\|_{h} \|G\|_{h}$ where $[\cdot , \cdot ]$ is the Lie bracket defined by  $[F,G]=FG-GF$,  and thus $(\mathcal{B}_{h,r}, [\cdot , \cdot ])$ is a Lie algebra.
\end{enumerate}
\end{proposition}

\begin{proof}
Proposition \ref{banach}\ref{item:subset} follows directly from the definition, so we only need to check the second one. Let $F,G\in \mathcal{B}_{h,r}$ with expansions
	\begin{equation*}
		F(x)=\sum_{{\bf k}\in \Gamma_{r}} \widehat{F}_{\bf k} {\rm e}^{{\rm i}\langle {\bf k},x \rangle}, \quad G(x)=\sum_{{\bf n}\in \Gamma_{r}} \widehat{G}_{\bf n} {\rm e}^{{\rm i}\langle {\bf n},x \rangle},
	\end{equation*}
	where $\widehat{F}_{\bf k},\widehat{G}_{\bf n} \in {\rm sl}(2,\mathbb{C})$ for any ${\bf n,k}\in \Gamma_{r}$. By direct calculation,
	\begin{equation}\label{fg}
		[F,G](x)=\sum_{{\bf k,n}\in \Gamma_{r}} [\widehat{F}_{\bf k},\widehat{G}_{\bf n}] {\rm e}^{{\rm i}\langle {\bf k+n},x \rangle}. 
	\end{equation} 
Since $[\![\mathbf{k}]\!]>r|{\bf k}|_{\eta}, [\![\mathbf{n}]\!]>r|{\bf n}|_{\eta}$,  we have $[\![{\bf k+n}]\!] >r(|{\bf k}|_{\eta}+|{\bf n}|_{\eta}) \geq r|{\bf k+n}|_{\eta}$, which means that ${\bf k+n}\in \Gamma_{r}$. On the other hand, rewrite \eqref{fg} as 	
\begin{equation*}
		[F, G](x) = \sum_{{\bf m}\in \Gamma_{r}} \Big(\sum_{\bf n+k=m} [\widehat{F}_{\bf k},\widehat{G}_{\bf n} ] \Big) {\rm e}^{{\rm i} \langle {\bf m},x \rangle}.
	\end{equation*}
	Then Proposition \ref{banach}\ref{item:closed} follows from $[\widehat{F}_{\bf k},\widehat{G}_{\bf n}] \in {\rm sl}(2,\mathbb{C})$,
	\begin{equation*}
	\begin{split}
		\|FG\|_{h}&=\sum_{\bf n} \Big(\sum_{\bf k} |\widehat{F}_{\bf k}\widehat{G}_{\bf n-k}|\Big) {\rm e}^{h|{\bf n}|_{\eta}}\\
		&\leq \sum_{\bf m} \Big(\sum_{\bf k} |\widehat{F}_{\bf k}| |\widehat{G}_{\bf m}|\Big) {\rm e}^{h|{\bf m+k}|_{\eta}}\\
		&\leq \Big(\sum_{\bf k} |\widehat{F}_{\bf k}| {\rm e}^{h|{\bf k}|_{\eta}}\Big) \Big(\sum_{\bf m} |\widehat{F}_{\bf m}| {\rm e}^{ h|{\bf m}|_{\eta}}\Big) =\|F\|_{h} \|G\|_{h},
	\end{split}
\end{equation*}	
and the same estimate on $\|GF\|_{h}$. 	
\end{proof}

\subsection{Non-resonance cancellation lemma}

We give a non-resonance cancellation lemma, which serves as the starting point of our proof. 
Let $A\in{\rm SL}(2,\mathbb{C})$, for any  $Y\in \mathcal{B}_{h,r}$,
we define the linear operator $L_{A}$ by $$(L_{A} Y)(x):= A^{-1}Y(x+\alpha)A-Y(x). $$ Suppose that $\mathcal{B}_{h,r} = \mathcal{B}_{h,r}^{\rm nre} (\sigma) \oplus \mathcal{B}_{h,r}^{\rm re}(\sigma)$, where $\mathcal{B}_{h,r}^{\rm nre}(\sigma)$ is the closed invariant subspace in $\mathcal{B}_{h,r}$ such that  $L_{A}$ restricted on $\mathcal{B}_{h,r}^{\rm nre}(\sigma)$ is invertible and
\begin{equation*}
\|L_{A}^{-1}\| \leq  \frac{1}{\sigma} \quad \text{on}\quad \mathcal{B}_{h,r}^{\rm nre}(\sigma).
\end{equation*}

In the following lemma, we prove that all non-resonant terms in the perturbation can be eliminated.  
\begin{lemma}[\cite{Cai2, Hou}]\label{hy}
	Let $d\in \mathbb{N}^{+}\cup \{\infty\}$, $h>0$, $r\in (0,1]$, $\alpha\in \mathbb{T}^{d}$, $\sigma>0$. Suppose that $A\in {\rm SL}(2,\mathbb{C})$, and $F\in \mathcal{B}_{h,r}$ with 
	\begin{equation*}
		\|F\|_{h}<\varepsilon<\min \{10^{-8}, \sigma^2\}.
	\end{equation*} Then  there exist $Y \in \mathcal{B}_{h,r}^{\rm nre}(\sigma)$ and $F^{\rm re}\in \mathcal{B}^{\rm re}_{h,r}(\sigma)$ such that ${\rm e}^{Y}$ conjugates the cocycle $(\alpha,A{\rm e}^{F})$ to $(\alpha, A{\rm e}^{F^{\rm re}})$, i.e.,
	\begin{equation*}
		{\rm e}^{-Y(x+\alpha)} A{\rm e}^{F(x)} {\rm e}^{Y(x)}= A{\rm e}^{F^{\rm re}(x)},
	\end{equation*}
	with $\|Y\|_{h} \leq \varepsilon^{\frac{1}{2}}$, $\|F^{\rm re}\|_{h}\leq 2\varepsilon$ and $\|F^{\rm re}-\mathbb{P}_\mathrm{re}F\|_{h} \leq 2\varepsilon^{\frac{4}{3}}$.
\end{lemma}

\subsection{One step of KAM iteration}

In this section, we give the one step of KAM iteration for $(\alpha, A{\rm e}^{F(x)})$ with $A\in \mathcal{M}\subset {\rm SL}(2,\mathbb{C})$ and ${F(x)} \in \mathcal{B}_{h,r}$, where
\begin{equation*}
	\mathcal{M} := \bigg\{\begin{pmatrix}
		{\rm e}^{{\rm i}\xi}&\zeta\\
		0&{\rm e}^{-{\rm i}\xi}
	\end{pmatrix}:\xi,\zeta \in \mathbb{C}\bigg\} \cup \bigg\{\begin{pmatrix}
		{\rm e}^{{\rm i}\xi}&0\\
		\zeta&{\rm e}^{-{\rm i}\xi}
	\end{pmatrix}:\xi,\zeta \in \mathbb{C} \bigg\}.
\end{equation*}

To eliminate the perturbation $F(x)$ in the cocycle, we need to deal with non-resonant case and resonant case separately. Here  we say $A$ is non-resonant up to $N$, denoted by $A\in \mathcal{NR}(N,\delta)$, if for any ${\bf k}\in \mathcal{T}_{N}\Gamma_{r}$, 
\begin{equation*}
	|{\rm e}^{{\rm i} (\langle{\bf k},\alpha \rangle \pm 2\xi)}-1| \geq \delta.
\end{equation*}
Otherwise, we say $A$ is resonant and denoted by $A\in \mathcal{RS}(N,\delta)$, which means there is a ${\bf k}^*\in \mathcal{T}_{N}\Gamma_{r}$ such that
\begin{equation*}
	|{\rm e}^{{\rm i} (\langle{\bf k}^*,\alpha \rangle+2\xi)}-1| < \delta \quad \text{or} \quad
	|{\rm e}^{{\rm i} (\langle{\bf k}^*,\alpha \rangle-2\xi)}-1| < \delta.
\end{equation*}
In the following subsection, we always fix $N=\frac{2|\log \varepsilon|}{h-h_{+}}$ where $h_{+}\in (0,h)$. Once we have these, we introduce the following key quantitative almost reducibility result, which gives the one step of KAM iteration.
	
\begin{proposition}\label{step}
Let $d\in \mathbb{N}^{+} \cup \{\infty\}$, $\eta>0$, $h>0$, $r\in (0,1]$, $\gamma\in (0,1)$, $\tau>1$, $\alpha\in {\rm DC}^{d}_{\gamma,\tau}$. Suppose that $A\in \mathcal{M}$ with $|\mathrm{Im}\xi| \leq \frac{1}{2}$, $F \in \mathcal{B}_{h,r}$, then for any $h_{+}\in (0,h)$,  $r_{+}\in (0,r)$,
	there exist $\varepsilon=\varepsilon(\eta,h,h_{+},r,r_{+},\gamma,\tau,|\zeta|)$, $c=c(\eta, \gamma,\tau)$ 
	such that if 
	\begin{equation}\label{small}
		\|F\|_{h} < \varepsilon< \frac{c}{(1+|\zeta|)^{10}}
		\min \bigg\{{\rm e}^{-(\frac{1}{h-h_{+}})^{\frac{10}{\eta}}}, {\rm e}^{-( \frac{1}{r-r_{+}})^{\frac{10}{\eta}}} \bigg\}, 
	\end{equation} 
then there exist $B\in C^{\omega}(2\mathbb{T}^{d}_{h}, {\rm SL}(2,\C))$, $A_{+}\in \mathcal{M}$, and $F_{+}\in \mathcal{B}_{h_{+},r_{+}}$ such that
\begin{equation*}
	B(x+\alpha)^{-1} A{\rm e}^{F(x)}  B(x) = A_{+}{\rm e}^{F_{+}(x)}.
\end{equation*}
Moreover, we have the following estimates:

\begin{itemize}
	\item{\bf Non-resonant case:} If $A\in \mathcal{NR}(N,\varepsilon^{\frac{1}{10}})$, then $B(\cdot)=\mathrm{e}^{Y(\cdot)}$ with
	\begin{equation*}
		\|Y\|_{h}\leq  \varepsilon^{\frac{1}{2}}, \quad \|F_{+}\|_{h_{+}} \leq 2\varepsilon^{3}, \quad A_{+}=A. 
	\end{equation*}

	\item {\bf Resonant case:} If $A\in \mathcal{RS}(N,\varepsilon^{\frac{1}{10}})$, then there exists  $\mathbf{k}^{*} \in \mathcal{T}_{N} \Gamma_{r}$ such that
	\begin{enumerate}[font=\normalfont, label={(\arabic*)}]
			\item $A_{+}$ takes the form 
		\begin{equation*}
			A_{+} = \begin{pmatrix}
				{\rm e}^{{\rm i}\xi_{+}} &\zeta_{+}\\ 0&{\rm e}^{-{\rm i} \xi_{+}}
			\end{pmatrix} \quad \text{or} \quad A_{+} = \begin{pmatrix}
				{\rm e}^{{\rm i}\xi_{+}} &0\\ \zeta_{+}&{\rm e}^{-{\rm i}\xi_{+}}
			\end{pmatrix},
		\end{equation*}
		where $\zeta_{+}\in \mathbb{C}$,  $\xi_{+}= \xi-\frac{\langle {\bf k}^{*},\alpha\rangle}{2}$ with estimates
				\begin{equation*}
			|\xi_{+}| \leq \varepsilon^{\frac{1}{10}}, \quad |\zeta_{+}| \leq \varepsilon^{\frac{9}{10}}.
		\end{equation*}

\item It holds that 
\begin{equation*}
	\|B\|_{0} \leq {\rm e}^{|\log \varepsilon|^{\frac{2}{2+\eta}}},   \quad \|F_{+}\|_{h_{+}} \leq \varepsilon^{100}.
\end{equation*}
	\end{enumerate}
\end{itemize}
\end{proposition}

\begin{proof} We distinguish the proof into two cases:

{\bf Case 1: Non-resonant case.}  Let $\sigma=\varepsilon^{\frac{1}{3}}$ and decompose  $\mathcal{B}_{h,r}$ as $\mathcal{B}^{\rm nre}_{h,r}(\sigma)\oplus \mathcal{B}^{\rm re}_{h,r}(\sigma)$, where 
	\begin{equation}\label{decomnr}
		\begin{split}
			&\mathcal{B}^{\rm nre}_{h,r}(\sigma)= \Big\{F\in \mathcal{B}_{h,r}:F(x)=\mathcal{T}_{N}F(x)\Big\},\\
			&\mathcal{B}^{\rm re}_{h,r}(\sigma) = \Big\{F\in \mathcal{B}_{h,r}:F(x)=\mathcal{R}_{N}F(x)\Big\}.
		\end{split}
	\end{equation}
	It is easy to see that 
	$\mathcal{B}_{h,r}^{\rm nre}(\sigma)$ is a closed invariant subspace of $\mathcal{B}_{h,r}$. Moreover, we have the following simple observation:
	\begin{lemma}\label{inverse1}
		The operator $L_{A}^{-1}: \mathcal{B}_{h,r}^{\rm nre}(\sigma)\rightarrow \mathcal{B}_{h,r}^{\rm nre}(\sigma)$ is bounded with $\|L_{A}^{-1}\|\leq  \frac{1}{\sigma}$.
	\end{lemma}
	\begin{proof}
		We only consider the case $A=\begin{pmatrix}
			{\rm e}^{{\rm i}\xi}&\zeta\\
			0&{\rm e}^{-{\rm i}\xi}
		\end{pmatrix}$, the proof for the case $A=\begin{pmatrix}
		{\rm e}^{{\rm i}\xi}&0\\
		\zeta&{\rm e}^{-{\rm i}\xi}
	\end{pmatrix}$ is similar. For any $F\in \mathcal{B}_{h,r}^{\rm nre}(\sigma)$, we only need to solve
		\begin{equation*}
			A^{-1}Y(x+\alpha)A-Y(x)  = F(x).
		\end{equation*}
		Expand $Y(x)=\sum_{\bf k} \widehat{Y}_{\bf k}{\rm e}^{{\rm i}\langle {\bf k},x\rangle}$ and $F(x)=\sum_{\bf k} \widehat{F}_{\bf k} {\rm e}^{{\rm i}\langle {\bf k},x\rangle}$ respectively. 
		Comparing the Fourier coefficients, one obtains that for ${\bf k}\in \Gamma_{r}$,
		\begin{equation}\label{solve}
			\begin{split}
				&\widehat{Y}^{2,1}_{\bf k} = \frac{\widehat{F}^{2,1}_{\bf k}}{{\rm e}^{{\rm i}(\langle {\bf k},\alpha\rangle+2\xi)}-1},\\
				&\widehat{Y}^{1,1}_{\bf k} =- \widehat{Y}^{2,2}_{\bf k} = \frac{\widehat{F}^{1,1}_{\bf k} + \zeta {\rm e}^{{\rm i}(\langle {\bf k},\alpha\rangle+\xi)} \widehat{Y}^{2,1}_{\bf k}}{{\rm e}^{{\rm i}\langle {\bf k},\alpha\rangle}-1},\\
				&\widehat{Y}^{1,2}_{\bf k} =\frac{\widehat{F}^{1,2}_{\bf k} + \zeta^{2} {\rm e}^{{\rm i}\langle {\bf k},\alpha\rangle} \widehat{Y}^{2,1}_{\bf k} -2 \zeta {\rm e}^{{\rm i}(\langle {\bf k},\alpha\rangle-\xi)} \widehat{Y}^{1,1}_{\bf k}}{{\rm e}^{{\rm i} (\langle {\bf k},\alpha\rangle-2\xi)}-1}.
			\end{split}	
		\end{equation}

		Recall the following estimate for $\alpha\in{\rm DC}^{d}_{\gamma,\tau}$:
		\begin{lemma}[Small denominators \cite{Proce}]\label{dioestimate}
			Let $d\in \mathbb{N}^{+}\cup \{\infty\}$, $\tau>1$, $\eta>0$, then for any ${\bf k} \in \mathbb{Z}^{d}_{*}$ we have the following estimate
			\begin{equation*}
				\sup_{{\bf k}\in \mathbb{Z}^{d}_{*}, |{\bf k}|_{\eta}\leq N} \prod_{j\in \mathbb{N}} (1+\langle j\rangle^{\tau} |{\bf k}_{j}|^{\tau}) \leq (1+N)^{C_{1}N^{\frac{1}{\eta+1}}},
			\end{equation*}
			where $C_{1}=C_{1}(\eta,\tau)$. Moreover,
			\begin{equation*}
				\prod_{j\in\mathbb{N}} (1+\langle j\rangle^{\tau} |{\bf k}_{j}|^{\tau}) \leq (1+|{\bf k}|_{\eta})^{C_{1}|{\bf k}|_{\eta}^{\frac{1}{\eta+1}}}.
			\end{equation*}
		\end{lemma}
		Combining Lemma \ref{dioestimate} with \eqref{small}, for any ${\bf k}\in\mathcal{T}_{N}\Gamma_{r}$ we have
		\begin{equation*}
			\|\langle {\bf k},\alpha\rangle\|_{\mathbb{T}}\geq \gamma (1+N)^{-C_{1} N^{\frac{1}{\eta+1}}} > \varepsilon^{\frac{1}{10}}.
		\end{equation*}
		Besides, it follows from $A\in \mathcal{NR}(N,\varepsilon^{\frac{1}{10}})$ that $|{\rm e}^{{\rm i} (\langle{\bf k},\alpha \rangle \pm 2\xi)}-1| \geq \varepsilon^{\frac{1}{10}}$ for any ${\bf k}\in\mathcal{T}_{N}\Gamma_{r}$.  Thus, the denominators in \eqref{solve} are well controlled and Lemma \ref{inverse1} follows.
	\end{proof}	
	
	By Lemma \ref{hy}, 
	there exist $Y \in \mathcal{B}_{h,r}^{\rm nre}(\sigma)$ and $F^{\rm re}\in \mathcal{B}^{\rm re}_{h,r}(\sigma)$ such that 
	\begin{equation*}
		{\rm e}^{-Y(x+\alpha)} A{\rm e}^{F(x)} {\rm e}^{Y(x)}= A{\rm e}^{F^{\rm re}(x)},
	\end{equation*}
	with the following estimates
	\begin{equation*}
		\|Y\|_{h}\leq \varepsilon^{\frac{1}{2}}, \quad \|F^{\rm re}\|_{h}\leq 2\varepsilon.
	\end{equation*}
	Let $B={\rm e}^{Y}$, $ A_{+}=A$, and $F_{+}(x)=F^{\rm re}(x)$. By the construction in \eqref{decomnr}, $F_+$ can be expressed as 
	\begin{equation*}
		F_{+}(x)=\sum_{{\bf k}\in \mathcal{R}_{N}\Gamma_{r}} \widehat{F}^{\rm re}_{\bf k} {\rm e}^{{\rm i}\langle {\bf k},x \rangle}.
	\end{equation*}  
	Therefore, for any $h_{+}\in (0,h)$, we have estimate 
	\begin{equation*}
		\begin{split}
			\| F_{+} \|_{h_{+}} = \sum_{{\bf k}\in \mathcal{R}_{N}\Gamma_{r}} \|\widehat{F}^{\rm re}_{\bf k}\| {\rm e}^{ h_{+}|{\bf k}|_{\eta}}
			&\leq {\rm e}^{-(h-h_{+}) N} \sum_{{\bf k}\in \mathcal{R}_{N}\Gamma_{r}} \|\widehat{F}^{\rm re}_{\bf k}\| {\rm e}^{ h|{\bf k}|_{\eta}}\\
			&\leq 2{\rm e}^{-(h-h_{+}) N} \|F\|_{h}< 2\varepsilon^{3},
		\end{split}		
	\end{equation*}
	where the last inequality follows from our choice that $N=\frac{2|\log \varepsilon|}{h-h_{+}}$.

	{\bf Case 2: Resonant case.} In view of $\alpha \in{\rm DC}^{d}_{\gamma,\tau}$  and  $A\in \mathcal{RS}(N,\varepsilon^{\frac{1}{10}})$, Lemma \ref{dioestimate} and \eqref{small} imply
\begin{equation} \label{real}
	\|2{\rm Re} \xi\|_{\mathbb{T}} > \|\langle {\bf k}^{*},\alpha \rangle\|_{\mathbb{T}} +2|{\rm Im} \xi|-\varepsilon^{\frac{1}{10}} \geq  \frac{\gamma}{2} (1+N)^{-C_{1} N^{\frac{1}{\eta+1}}},
\end{equation}
as a consequence, 
\begin{equation}\label{unique3}
	\Big(|{\rm e}^{{\rm i} (\langle{\bf k^{*}},\alpha \rangle + 2\xi)}-1|-\varepsilon^{\frac{1}{10}} \Big) \cdot	\Big(|{\rm e}^{{\rm i} (\langle{\bf k^{*}},\alpha \rangle - 2\xi)}-1|-\varepsilon^{\frac{1}{10}} \Big)<0,
\end{equation}
which shows that the concept of resonance is well-defined.  In fact, if $|{\rm e}^{{\rm i} (\langle{\bf k}^*,\alpha \rangle-2\xi)}-1| < \varepsilon^{\frac{1}{10}}$, then \eqref{unique3} directly follows from \eqref{real} that
\begin{equation*}
	|{\rm e}^{{\rm i} (\langle{\bf k}^*,\alpha \rangle+2\xi)}-1|=\|\langle{\bf k}^*,\alpha \rangle-2\xi\|_{\mathbb{T}} \geq |4\xi|-\varepsilon^{\frac{1}{10}} \gg \varepsilon^{\frac{1}{10}}.
\end{equation*}

Note that \eqref{real} also implies that in the resonant case $\|2{\rm Re} \xi\|_{\mathbb{T}}$ always has a lower bound, which allows us to diagonalize the constant matrix $A$.  Just assume $A=\begin{pmatrix}
	{\rm e}^{{\rm i}\xi}&\zeta\\
	0&{\rm e}^{-{\rm i}\xi}
\end{pmatrix}$, then there exists $P= \begin{pmatrix}
		1&\frac{\zeta}{{\rm e}^{-{\rm i} \xi}-{\rm e}^{{\rm i}\xi}}\\0&1
	\end{pmatrix}$, such that 
$$
P^{-1} AP= \begin{pmatrix}
	{\rm e}^{{\rm i}\xi}&0\\
	0&{\rm e}^{-{\rm i}\xi}
\end{pmatrix}= \tilde{A}.$$
Moreover, just note \begin{equation*}
	\begin{split}
		|{\rm e}^{-{\rm i} \xi}-{\rm e}^{{\rm i}\xi}| &= |\cos{\rm Re}\xi\cdot ({\rm e}^{{\rm Im}\xi}-{\rm e}^{-{\rm Im}\xi}) -{\rm i}\sin {\rm Re}\xi \cdot({\rm e}^{{\rm Im}\xi}+{\rm e}^{-{\rm Im}\xi})|\\
		&\geq \frac{1}{4}\|{2\rm Re}\xi\|_{\mathbb{T}},
	\end{split}
\end{equation*}
then we have estimate
\begin{equation*}
	\|P\|\leq 1+\frac{4|\zeta|}{\|2{\rm Re}\xi\|_{\mathbb{T}}} \leq 1+ \frac{8|\zeta|}{\gamma}(1+N)^{C_{1} N^{\frac{1}{\eta+1}}} \leq \frac{1}{2} {\rm e}^{|\log \varepsilon|^{\frac{2}{2+\eta}}}. 
\end{equation*}
Moreover,  $P^{-1} A{\rm e}^{F(x)}P=\tilde{A}{\rm e}^{\tilde{F}(x)},$
where $\tilde{F}=P^{-1}FP\in \mathcal{B}_{h,r}$ satisfies 
\begin{equation}\label{newerror}
	\|\tilde{F}\|_{h}\leq \|F\|_{h} \|P\|^{2} \leq  {\rm e}^{2|\log \varepsilon|^{\frac{2}{2+\eta}}} \varepsilon =:\tilde{\varepsilon}.
\end{equation}
By the choice of $\varepsilon$ in \eqref{small} we have $\tilde{\varepsilon}\leq \varepsilon^{\frac{9}{10}}$.

After the diagonalization, we are ready to solve the non-resonant terms of the perturbation.  For this purpose, we need to analyze the fine structure of the small denominators. We just consider the case  
\begin{equation}\label{res-u}
|{\rm e}^{{\rm i} (\langle{\bf k}^*,\alpha \rangle-2\xi)}-1| < \varepsilon^{\frac{1}{10}},
\end{equation} since the other case can be dealt with similarly. 
The following lemma shows that the integer cone $\Gamma_{r}$ implies that the resonant site in $\mathcal{T}_{N'}\Gamma_{r}$ is unique under the proper truncation $N' \gg N$.

\begin{lemma}[Uniqueness]\label{unique}
	Let $N'=C_{2} |\log\varepsilon|^{1+\frac{\eta}{2}}-N$ and $C_{3}=\frac{1}{10}C_{2}^{-\frac{2}{2+\eta}}$, where $C_{2}=C_{2}(\eta,\gamma,\tau)$ is the constant such that 
	\begin{equation}\label{C2}
		\frac{1}{10}\Big(\frac{x}{C_{2}}\Big)^{\frac{1}{1+\frac{\eta}{2}}} \geq -\log(\frac{\gamma}{2}) +C_{1} x^{\frac{1}{1+\eta}} \log (1+x), \ \forall \ x>0.
	\end{equation}
	Then for any ${\bf k}\in \mathcal{T}_{N'}\Gamma_{r}$ we have
	\begin{align}
		&|{\rm e}^{{\rm i} \langle{\bf k},\alpha \rangle}-1| \geq \varepsilon^{\frac{1}{10}}, \label{unique1}\\
		&|{\rm e}^{{\rm i} (\langle{\bf k},\alpha \rangle \pm 2\xi)}-1| \geq \varepsilon^{\frac{1}{10}}, \ {\text when} \ {\bf k}\neq {\bf k}^{*}. \label{unique2}
	\end{align}
\end{lemma}
\begin{proof}
	If \eqref{unique1} does not hold, then by using $\alpha\in {\rm DC}^{d}_{\gamma,\tau}$, Lemma  \ref{dioestimate} and \eqref{C2},
	\begin{equation*}
			\varepsilon^{\frac{1}{10}} > \|\langle {\bf k},\alpha \rangle\|_{\mathbb{T}} \geq \gamma (1+|{\bf k}|_{\eta})^{-C_{1}|{\bf k}|_{\eta}^{\frac{1}{\eta+1}}} \geq \mathrm{e}^{-C_{3}|\mathbf{k}|_{\eta}^{\frac{2}{2+\eta}} }.
	\end{equation*}	
Thus, combining the above inequality with the choice of $N'$, we have 
	\begin{equation}\label{xiajie}
		|{\bf k}|_{\eta} > C_{2}|\log\varepsilon|^{1+\frac{\eta}{2}} >N',
	\end{equation}
	which shows a contradiction to ${\bf k}\in \mathcal{T}_{N'}\Gamma_{r}$. 
	
 If \eqref{unique2} does not hold, then there exists ${\bf k' \neq k^{*}}$ such that $|{\rm e}^{{\rm i} (\langle{\bf k}',\alpha \rangle+2\xi)}-1| < \varepsilon^{\frac{1}{10}}$ or $|{\rm e}^{{\rm i} (\langle{\bf k}',\alpha \rangle-2\xi)}-1| < \varepsilon^{\frac{1}{10}}$. This implies that 
	\begin{equation*}
		2\varepsilon^{\frac{1}{10}} > \max\{\|\langle {\bf k}',\alpha\rangle + 2\xi + (\langle {\bf k^{*}},\alpha\rangle -2\xi)  \|_{\mathbb{T}}, \|\langle {\bf k}',\alpha\rangle - 2\xi - (\langle {\bf k^{*}},\alpha\rangle - 2\xi)  \|_{\mathbb{T}} \}.
	\end{equation*}
	Since ${\bf k}'\in \Gamma_{r}$,  it follows from the structure of the integer cone $\Gamma_{r}$ that  
	\begin{equation*}
[\![{\bf k'+k^{*}}]\!] \geq r(|{\bf k'}|_{\eta}+|{\bf k^{*}}|_{\eta})>0,
	\end{equation*}
	which implies that ${\bf k}'+{\bf k}^{*} \neq {\bf 0}$. Moreover, by $\alpha\in {\rm DC}^{d}_{\gamma,\tau}$ and Lemma  \ref{dioestimate},
	\begin{equation*}
		2\varepsilon^{\frac{1}{10}}> \|\langle {\bf k}'\mp{\bf k}^{*},\alpha\rangle\|_{\mathbb{T}} \geq \gamma (1+|{\bf k}'\mp{\bf k}^{*}|_{\eta})^{-C_{1}|{\bf k}'\mp{\bf k}^{*}|_{\eta}^{\frac{1}{\eta+1}}}.
	\end{equation*}
	Same as \eqref{xiajie},  the inequality \eqref{C2} would imply that  
	\begin{equation*}
		|{\bf k}'\mp{\bf k}^{*}|_{\eta} > C_{2}|\log \varepsilon|^{1+\frac{\eta}{2}},
	\end{equation*}
and consequently
	\begin{equation*}
		|{\bf k}'|_{\eta}>C_{2} |\log\varepsilon|^{1+\frac{\eta}{2}}- N =N'.
	\end{equation*}
	This contradicts to ${\bf k}'\in \mathcal{T}_{N'}\Gamma_{r}$, and thus we finish the proof. 
\end{proof}

  Let $\sigma=\tilde{\varepsilon}^{\frac{1}{3}}$ and rewrite the ${\bf k}$-th Fourier coefficient of $\tilde{F}$ by $\widehat{F}_{\bf k} = \begin{pmatrix}
	a_{\bf k}&b_{\bf k}\\c_{\bf k}&-a_{\bf k}
\end{pmatrix}$ for any $\tilde{F}\in \mathcal{B}_{h,r}$. By \eqref{res-u} and Lemma \ref{unique}, the space decomposition with respect to $\tilde{A},\sigma$ takes the form as:
\begin{equation*}
	\begin{split}
		&\mathcal{B}^{\rm nre}_{h,r}(\sigma)= \bigg\{ \tilde{F}(x)=\mathcal{T}_{N'} \tilde{F}(x)-\begin{pmatrix}
			0&b_{\bf k^{*}}	\\0&0
		\end{pmatrix}{\rm e}^{{\rm i}\langle {\bf k}^{*}, x\rangle}  \bigg\},\\
		&\mathcal{B}^{\rm re}_{h,r}(\sigma)= \bigg\{ \tilde{F}(x)=\mathcal{R}_{N'} \tilde{F}(x)
		+\begin{pmatrix}
			0&b_{\bf k^{*}}	\\0&0
		\end{pmatrix}{\rm e}^{{\rm i}\langle {\bf k}^{*}, x\rangle} \bigg\}.
	\end{split}
\end{equation*}
It follows directly that 
$\mathcal{B}_{h,r}^{\rm nre}(\sigma)$ is a closed invariant subspace of $\mathcal{B}_{h,r}$. Moreover, we have the following:

\begin{lemma}\label{inverse2}
	The operator $L_{\tilde{A}}^{-1}: \mathcal{B}_{h,r}^{\rm nre}(\sigma)\rightarrow \mathcal{B}_{h,r}^{\rm nre}(\sigma)$ is bounded with  $\|L_{\tilde{A}}^{-1}\|\leq  \frac{1}{\sigma}$.
\end{lemma}
\begin{proof}
	For any $\tilde{F}\in \mathcal{B}_{h,r}^{\rm nre}(\sigma)$, in order to solve  $\tilde{F}(x)=L_{\tilde{A}}Y(x)$, we only need to 
	expand $Y(x)=\sum_{\bf k} \widehat{Y}_{\bf k}{\rm e}^{{\rm i}\langle {\bf k},x\rangle}$ and $\tilde{F}(x)=\sum_{\bf k} \widehat{F}_{\bf k} {\rm e}^{{\rm i}\langle {\bf k},x\rangle}$ respectively. 
	Direct calculation shows 
	\begin{equation*}
		\begin{split}
			&\widehat{Y}_{\bf k^*} = \begin{pmatrix}
				a_{\bf k^*}/ ({\rm e}^{ {\rm i}\langle {\bf k^*},\alpha\rangle}-1)&0\\c_{\bf k^*}/({\rm e}^{{\rm i}(\langle {\bf k^*},\alpha\rangle+2\xi)}-1)&-a_{\bf k^*}/({\rm e}^{{\rm i}\langle {\bf k^*},\alpha\rangle}-1)
			\end{pmatrix}, \\
			&\widehat{Y}_{\bf k} = \begin{pmatrix}
				a_{\bf k}/ ({\rm e}^{ {\rm i}\langle {\bf k},\alpha\rangle}-1)&b_{\bf k}/({\rm e}^{{\rm i}(\langle {\bf k},\alpha\rangle-2\xi)}-1)\\c_{\bf k}/({\rm e}^{{\rm i}(\langle {\bf k},\alpha\rangle+2\xi)}-1)&-a_{\bf k}/({\rm e}^{{\rm i}\langle {\bf k},\alpha\rangle}-1)
			\end{pmatrix}, \ \forall  \ {\bf k}\neq {\bf k}^{*}.
		\end{split}
	\end{equation*}
 Then the result follows from Lemma \ref{unique}.
\end{proof}

Once we have Lemma \ref{inverse2},  we then apply Lemma \ref{hy} to obtain $Y \in \mathcal{B}_{h,r}^{\rm nre}(\sigma)$ and $F^{\rm re}\in \mathcal{B}^{\rm re}_{h,r}(\sigma)$ such that 
\begin{equation*}
	{\rm e}^{-Y(x+\alpha)} \tilde{A}{\rm e}^{\tilde{F}(x)} {\rm e}^{Y(x)}= \tilde{A}{\rm e}^{F^{\rm re}(x)},
\end{equation*}
with the following estimates
\begin{equation*}
	\|Y\|_{h}\leq \tilde{\varepsilon}^{\frac{1}{2}}\leq\varepsilon^{\frac{9}{20}}, \quad \|F^{\rm re}\|_{h}\leq 2\tilde{\varepsilon}\leq 2\varepsilon^{\frac{9}{10}}.
\end{equation*}
Next, the resonant term $\begin{pmatrix}
	0&b_{\bf k^{*}}\\0&0
\end{pmatrix}{\rm e}^{{\rm i}\langle {\bf k}^{*}, x\rangle}$ in $F^{\rm re}(x)$ can be eliminated by the rotation conjugation $Q_{{\bf k}^{*}}(x)$, where
\begin{equation*}
	Q_{\bf k}(x):= R_{\frac{\langle {\bf k},x\rangle}{2}} =\begin{pmatrix}
		{\rm e}^{\frac{{\rm i}}{2}\langle {\bf k},x\rangle} &0\\
		0&{\rm e}^{-\frac{{\rm i}}{2}\langle {\bf k},x\rangle} 
	\end{pmatrix},
\end{equation*}
which is defined on $2\mathbb{T}^{d}$. Indeed,  direct calculation shows that
\begin{equation*}
 Q_{{\bf k}^{*}}(x+\alpha)^{-1}\tilde{A} Q_{{\bf k}^{*}}(x) =\begin{pmatrix}
			{\rm e}^{{\rm i}(\xi -\frac{\langle {\bf k}^{*},\alpha \rangle}{2})}&0\\0&{\rm e}^{-{\rm i} (\xi-\frac{\langle {\bf k}^{*},\alpha \rangle}{2})} 
		\end{pmatrix}=: \bar{A},
\end{equation*}
and
\begin{eqnarray*}		
  Q_{{\bf k}^{*}}(x)^{-1} F^{\rm re}(x) Q_{{\bf k}^{*}}(x) &=& \begin{pmatrix}
			0&b_{\bf k^{*}}\\0&0
		\end{pmatrix} + Q_{-{\bf k}^{*}}(x) \mathcal{R}_{N'}F^{\rm re}(x) Q_{{\bf k}^{*}}(x)\\
		&=:&L + G(x)
		=:\bar{F}(x).	
\end{eqnarray*}	
Let $B=P\cdot {\rm e}^{Y}\cdot Q_{{\bf k}^{*}} \in C^{\omega}(2\mathbb{T}^{d}_{h}, {\rm SL}(2,\mathbb{C}))$, then \begin{equation*}
	B(x+\alpha)^{-1}A{\rm e}^{F(x)} B(x)={\rm e}^{\bar{A}}{\rm e}^{\bar{F}(x)},
\end{equation*}
with  estimate $$\|B(x)\|_{0} \leq \|P\| \cdot  \|{\rm e}^{Y(x)}\|_{0} \leq {\rm e}^{|\log \varepsilon|^{\frac{2}{2+\eta}}}.$$
Rewrite the cocycle as
\begin{eqnarray*}
	\bar{A}{\rm e}^{\bar{F}(x)} = \bar{A}{\rm e}^{L}{\rm e}^{-L}{\rm e}^{\bar{F}(x)} 
	=\begin{pmatrix}
		{\rm e}^{{\rm i}\xi_{+}}&\zeta_{+}\\0&{\rm e}^{-{\rm i}\xi_{+}}
	\end{pmatrix} {\rm e}^{F_{+}(x)}
	=: A_{+} {\rm e}^{F_{+}(x)},
\end{eqnarray*}
where $\xi_{+}= \xi-\frac{\langle {\bf k}^{*},\alpha\rangle}{2}$, $\zeta_{+}=b_{\bf k^{*}}{\rm e}^{{\rm i}\xi_{+}}$. By the decay of Fourier coefficient $|b_{\bf k^{*}}|\leq \|F^{\rm re}\|_{h} {\rm e}^{-h|{\bf k}^{*}|_{\eta}}$ and \eqref{newerror}, it follows that
\begin{equation*}
	|\zeta_{+}|\leq \|F^{\rm re}\|_{h}  {\rm e}^{-h|{\bf k}^{*}|_{\eta}}{\rm e}^{\varepsilon^{\frac{1}{10}}}\leq \varepsilon^{\frac{9}{10}}.
\end{equation*}

Furthermore, by  Baker-Campbell-Hausdorff Formula,  we have
\begin{equation}\label{F+}
	F_{+}(x) = G(x) +\frac{1}{2} \left[-L,G(x)\right] + \frac{1}{12}    \left[-L,\left[-L,G(x)\right] \right] +  \cdots.
\end{equation}
The following result is important for us, which says that the rotation $Q_{{\bf k}^{*}}(x)$ preserves the cone structure, at the cost of shrinking $r$ a little, as shown in Figure \ref{iteratecone}. Consequently,  $F_{+}$ also has the cone structure. This is the key step why this modified KAM iteration can be iterated. 

\begin{figure}[htbp]
	\centering
	\begin{tikzpicture}[tdplot_main_coords]	
		\coordinate (O) at (0,0,0);
		\draw (0,0,-0.5) -- (O);
		\foreach \x in {-4,-3,...,4}
		\foreach \y in {-4,-3,...,4}
		{
			\draw[gray,very thin]  (\x,-4) -- (\x,4);
			\draw[gray,very thin]  (-4,\y) -- (4,\y);
		}
		\draw[->] (-6,0,0) -- (6,0,0) node[right]{};
		\draw[->] (0,6,0) -- (0,-6,0) node[right]{};
		\coneback[surface]{2}{3}{25}
		\coneback[surface]{2.5}{2.5}{15}
		\coneback[surface]{3}{2}{10}
		\draw[->] (O) -- (0,0,4) node[above]{};
		\conefront[surface]{3}{2}{10}
		\conefront[surface]{2.5}{2.5}{15}
		\conefront[surface]{2}{3}{25}
		\draw[->] (2.5,2.5,2)node[right]{$\Gamma_{r_{n+1}}$} -- (2.2,2.2,2)node{};
		\draw[->] (2.15,2.15,2.5)node[right]{$\Gamma_{r_{n}}$} -- (1.85,1.85,2.5)node{};
		\draw[->] (1.8,1.8,3)node[right]{$\Gamma_{r_{n-1}}$} -- (1.5,1.5,3)node{};
		\draw (2.6,2.6,1.5) -- (2.6,2.6,1.5) node[below] {$r_{n+1}<r_{n}<r_{n-1}$};
	\end{tikzpicture}
	\caption{Integer cones in the KAM iteration}
	\label{iteratecone}
\end{figure}
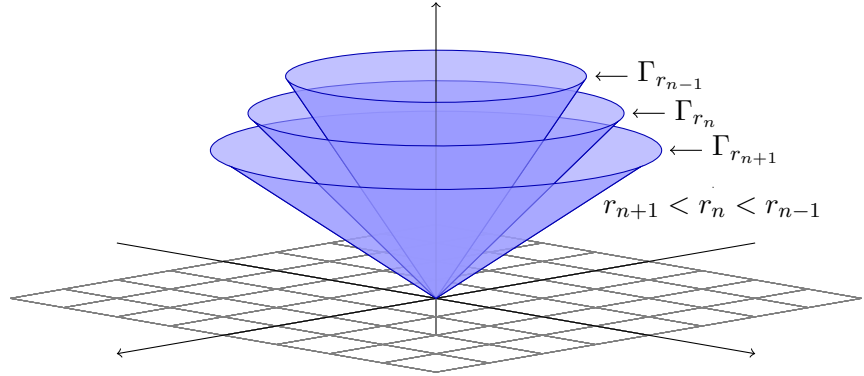

\begin{lemma}\label{keep}
	For any $F^{\rm re}\in \mathcal{B}^{\rm re}_{h_{+},r}(\sigma)$ and ${\bf k}^{*}\in \mathcal{T}_{N}\Gamma_{r}$, we have
	\begin{equation*}
		G(x)=Q_{-{\bf k}^{*}}(x)\cdot ( \mathcal{R}_{N'} F^{\rm re}(x)) \cdot Q_{{\bf k}^{*}}(x) \in \mathcal{B}_{h_{+},r_{+}}.
	\end{equation*}	
	Consequently, we have $F_{+}\in \mathcal{B}_{h_{+},r_{+}}$.
\end{lemma}
\begin{proof}
	Since $\mathcal{R}_{N'}F^{\rm re}\in \mathcal{B}^{\rm re}_{h_{+},r}$, then the $\bf k$-th term in its Fourier series is 
	\begin{equation*}
		\widehat{F}_{\bf k}{\rm e}^{{\rm i} \langle {\bf k}, x\rangle} = \begin{pmatrix}
			a_{\bf k}&b_{\bf k}\\
			c_{\bf k}&-a_{\bf k}
		\end{pmatrix}{\rm e}^{{\rm i}\langle {\bf k}, x\rangle}, \quad \forall \ {\bf k}\in \mathcal{R}_{N'}\Gamma_{r}.
	\end{equation*}
	Let $D_{\bf k}(x) = Q_{-\bf k^{*}}(x)\widehat{F}_{\bf k}{\rm e}^{{\rm i}\langle {\bf k}, x\rangle} Q_{\bf k^{*}}(x) $. By direct calculation we have
	\begin{equation*}
		D_{\bf k}(x) = \begin{pmatrix}
			a_{\bf k}&0\\0&-a_{\bf k}
		\end{pmatrix} {\rm e}^{{\rm i} \langle {\bf k}, x\rangle} + \begin{pmatrix}
			0&b_{\bf k}\\0&0
		\end{pmatrix}{\rm e}^{{\rm i}\langle {\bf k-k^{*}}, x\rangle}+ \begin{pmatrix}
			0&0\\c_{\bf k}&0
		\end{pmatrix} {\rm e}^{{\rm i}\langle {\bf k+k^{*}}, x\rangle}.
	\end{equation*}
	On the one hand, since  ${\bf k^{*}} \in \mathcal{T}_{N}\Gamma_{r}$ and ${\bf k}\in  \mathcal{R}_{N'}\Gamma_{r}$, we conclude that 
	\begin{equation*}
		\begin{split}
			[\![\bf k-k^{*}]\!] &>r|{\bf k}|_{\eta} -[\![\bf k^{*}]\!]\\
			&\geq r|{\bf k}|_{\eta} -N\\
			&= r_{+}|{\bf k}|_{\eta} +(r-r_{+})|{\bf k}|_{\eta} -N\\
			&\geq  r_{+}|{\bf k}|_{\eta} +(r-r_{+})\Big(C_{2}|\log \varepsilon|^{1+\frac{\eta}{2}}-N\Big) -N\\
			&\geq r_{+}|{\bf k}|_{\eta}  + C_{2}|\log \varepsilon|^{-\frac{\eta}{10}} |\log \varepsilon|^{1+\frac{\eta}{2}}-2N\\
			&\geq r_{+}|{\bf k}|_{\eta} +N\\
			&\geq r_{+}|{\bf k}|_{\eta} + r_{+} |{\bf k^{*}}|_{\eta}\\
			&\geq r_{+}|{\bf k-k^{*}}|_{\eta},
		\end{split}
	\end{equation*}
	where we use the fact $N\leq |\log \varepsilon|^{1+\frac{\eta}{8}}$ and $|\log\varepsilon| \geq (r-r_{+})^{-\frac{10}{\eta}}$. This just means ${\bf k-k^{*}} \in \Gamma_{r_{+}}$. On the other hand, $[\![{\bf k+k^{*}}]\!]\geq r(|{\bf k}|_{\eta}+|{\bf k^{*}}|_{\eta}) > r_{+} |{\bf k+k^{*}}|_{\eta}$ means ${\bf k+k^{*}}\in \Gamma_{r_{+}}$. We conclude that ${\bf k},{\bf k-k^{*}}, {\bf k+k^{*}}\in \Gamma_{r_{+}}$ and thus $D_{\bf k}(x)\in \mathcal{B}_{h_{+},r_{+}}$.

By Proposition \ref{banach}, we have $
		G(x)=\sum_{{\bf k}\in \mathcal{R}_{N'}\Gamma_{r_{+}}} D_{\bf k}(x)\in\mathcal{B}_{h_{+},r_{+}}.$
	Rewrite $G=\begin{pmatrix}
	G_{11}&G_{12}\\G_{21}&-G_{11}
\end{pmatrix}$, then
\begin{equation*}
[-L,G(x)]=\left[\begin{pmatrix}
		0&-b_{\bf k^{*}}\\0&0
	\end{pmatrix},G\right] = \begin{pmatrix}
	-b_{\bf k^{*}} G_{21} & 2b_{\bf k^{*}} G_{11}\\0&b_{\bf k^{*}} G_{21}
\end{pmatrix} \in \mathcal{B}_{h_{+},r_{+}},
\end{equation*}
which implies that the R.H.S. in \eqref{F+} belongs to $\mathcal{B}_{h_{+},r_{+}}$. Therefore, we finish the proof again by Proposition \ref{banach}.
\end{proof}

By \eqref{F+} and Lemma \ref{keep}, we have 
\begin{equation*}
	\|F_{+}\|_{h_{+}} \leq 2\|G\|_{h_{+}} \leq 2\| \mathcal{R}_{N'}F^{\rm re}(x)\|_{h_{+}} \|Q_{{\bf k}^{*}}\|_{h_{+}}^{2} \leq  4\tilde{\varepsilon}{\rm e}^{-N'(h-h_{+})} {\rm e}^{ h_{+}|{\bf k}^{*}|_{\eta}}.
\end{equation*}
Since $N'=C_{2} |\log \varepsilon|^{1+\frac{\eta}{2}}-N \gg |\log \varepsilon|^{\frac{\eta}{10}} N$ and ${\bf k}^{*}\in \mathcal{T}_{N}\Gamma_{r}$, one can get that
\begin{equation*}
	\|F_{+}\|_{h_{+}} \leq   \varepsilon^{\frac{9}{10}} {\rm e}^{-100N(h-h_{+})}  \varepsilon^{\frac{-h_{+}}{h-h_{+}}} \leq \varepsilon^{100}.
\end{equation*}
This finishes the proof.
\end{proof}

\section{Reducibility of almost-periodic cocycles}

By the KAM iteration developed in the last section, we now prove the reducibility results of the almost-periodic cocycle $(\alpha, A\mathrm{e}^{F})$ with perturbation $F\in\mathcal{B}_{h,r}$. We always choose 
\begin{equation*}
	A_{0}=A, \ \  F_{0}=F, \ \  \varepsilon_{0}=\varepsilon, \ \ h_{0}=h, \ \  r_{0}=r.
\end{equation*}
For $n\geq 0$, we define the sequences
\begin{equation}\label{star}
	\varepsilon_{n+1}=2\varepsilon_{n}^{3},\
	h_{n+1}= h_{n}- \frac{h-h'}{(n+2)^{2}},\
	r_{n+1}= r_{n}- \frac{r-r'}{(n+2)^{2}},\
	N_{n} = \frac{2|\log \varepsilon_{n}| }{h_{n}-h_{n+1}}.\tag{$*$}
\end{equation}
Two situations need to be treated separately for $A\in \mathcal{M}$, i.e.,  the eigenvalues of $A$ are $\mathrm{e}^{\pm\mathrm{i}\rho}$ with $\rho\in \mathbb{R}$ (elliptic case and parabolic case) or $\mathrm{e}^{\pm \mathrm{i}\xi} $ with $\xi\notin \mathbb{R}$ (hyperbolic case).

\subsection{Elliptic case and parabolic case}
Suppose that $A=\begin{pmatrix}
	{\rm e}^{{\rm i}\rho}&\zeta\\0&{\rm e}^{-{\rm i}\rho}
\end{pmatrix}$ with $\rho\in \mathbb{R}$, the following Proposition \ref{elliparabolic} shows that $(\alpha, A\mathrm{e}^{F})$ is almost reducible.

\begin{proposition}\label{elliparabolic}
Let $d\in \mathbb{N}^{+} \cup \{\infty\}$, $\eta>0$, $h>0$, $h'\in (0,h)$, $r\in (0,1]$, $r'\in (0,r)$, $\gamma\in(0,1)$, $\tau>1$, $\alpha\in {\rm DC}^{d}_{\gamma,\tau}$. Suppose that $F\in \mathcal{B}_{h,r}$.  There exists $\varepsilon=\varepsilon(\eta,h,h',r,r',\gamma,\tau,|\zeta|)$, $c=c(\eta,\gamma,\tau)$ such that if
	\begin{equation*}
		\|F\|_{h} < \varepsilon< \frac{c}{(1+|\zeta|)^{10}} \min \bigg\{{\rm e}^{-(\frac{1}{h-h'})^{\frac{10}{\eta}}}, {\rm e}^{-(\frac{1}{r-r'})^{\frac{10}{\eta}}} \bigg\},
	\end{equation*} 
then there exist $\Phi_{n}\in C^{\omega}(2\mathbb{T}^{d}_{h_{n}}, \mathrm{SL}(2,\mathbb{C}))$ with $\|\Phi_{n}\|_{0}\leq{\rm e}^{C_{\eta}|\log \varepsilon_{n}|^{\frac{2}{2+\eta}}}$,   $C_{\eta}:=(2^{\frac{2}{2+\eta}}-1)^{-1}$, and $F_{n}\in \mathcal{B}_{h_{n},r_{n}}$ with $\|F_{n}\|_{h_{n}} \leq \varepsilon_{n}$ such that
\begin{equation*}
	\Phi_{n}(x+\alpha)^{-1} A\mathrm{e}^{F(x)} \Phi_{n}(x) = A_{n}\mathrm{e}^{F_{n}(x)},
\end{equation*}
where $A_{n}=\begin{pmatrix}
	\mathrm{e}^{\mathrm{i}\rho_{n}} & \zeta_{n}\\
	0&\mathrm{e}^{-\mathrm{i}\rho_{n}}
\end{pmatrix} \ \text{or} \ A_{n}=\begin{pmatrix}
	\mathrm{e}^{\mathrm{i}\rho_{n}} &0 \\
	\zeta_{n}&\mathrm{e}^{-\mathrm{i}\rho_{n}}
\end{pmatrix}$ with $\rho_{n}\in \mathbb{R}$ and $|\zeta_{n}|<|\zeta|$. 

Moreover, if we denote $\Theta_{n} = \cup_{\mathbf{k}\in \mathcal{T}_{N_{n}} \Gamma_{r_{n}}}\Theta_{n}(\mathbf{k})$, where 
\begin{equation*}
	\Theta_{n}(\mathbf{k}) = \Big\{ \rho\in \mathbb{R}: |{\rm e}^{{\rm i} (\langle{\bf k},\alpha \rangle+2\rho_{n})}-1| \leq \varepsilon_{n}^{\frac{1}{10}}\Big\} \cup \Big\{ \rho\in \mathbb{R}: |{\rm e}^{{\rm i} (\langle{\bf k},\alpha \rangle-2\rho_{n})}-1| \leq \varepsilon_{n}^{\frac{1}{10}} \Big\},
\end{equation*}
then we have the following:
	\begin{enumerate}[font=\normalfont, label={(\arabic*)}]
		\item \label{item:nonres}If $\rho\notin \Theta_{n}$, then $ \Phi_{n+1} = \Phi_{n} \cdot \mathrm{e}^{Y_{n}}$ with 
		\begin{equation*} \|Y_{n}\|_{h}\leq \varepsilon_{n}^{\frac{1}{2}}, \quad \rho_{n+1}=\rho_{n}, \quad \zeta_{n+1}=\zeta_{n}.	
		\end{equation*}	
		
	\item \label{item:reson}If  $\rho\in \Theta_{n}(\mathbf{k}^{*}_{n})$, then  $\Phi_{n+1} = \Phi_{n} \cdot B_{n} $ with 
	\begin{equation*}
		\|B_{n}\|_{0}\leq  {\rm e}^{|\log \varepsilon_{n}|^{\frac{2}{2+\eta}}}, \quad \rho_{n+1} =\rho_{n}-\frac{\langle \mathbf{k}^{*}_{n},\alpha \rangle}{2}, \quad |\rho_{n+1}|\leq \varepsilon_{n}^{\frac{1}{10}}, \quad |\zeta_{n+1}|\leq \varepsilon_{n}^{\frac{9}{10}}.
	\end{equation*}

	\item \label{item:spe}If $\rho\in \Theta_{n_{j}}(\mathbf{k}^{*}_{n_{j}})\cap \Theta_{n_{j+1}}(\mathbf{k}^{*}_{n_{j+1}})$, then 
	\begin{equation*}
		|\mathbf{k}^{*}_{n_{j+1}}|_{\eta} \geq |{\bf k}^{*}_{n_{j}}|_{\eta}^{1+\frac{\eta}{4+\eta}}.
	\end{equation*}
\end{enumerate}
\end{proposition}

\begin{proof}
We are going to prove Proposition \ref{elliparabolic} inductively. Suppose that we are at $n$-th step, i.e., we already construct $\Phi_{n}$ such that
\begin{equation*}
	\Phi_{n}(x+\alpha)^{-1} A \mathrm{e}^{F(x)} \Phi_{n}(x) = A_{n} \mathrm{e}^{F_{n}(x)},
\end{equation*}
with following estimates
\begin{equation*}
	\|\Phi_{n}\|_{0}\leq {\rm e}^{C_{\eta}|\log \varepsilon_{n}|^{\frac{2}{2+\eta}}},  \quad \|F_{n}\|_{h_{n}} \leq \varepsilon_{n}, \quad \rho_{n}\in \mathbb{R}, \quad |\zeta_{n}|\leq |\zeta|.
\end{equation*}
By the selection of \eqref{star} and $|\zeta_{n}|\leq |\zeta|$, for any $n\geq 0$ we have
	\begin{equation*}
		\varepsilon_{n} < \frac{c}{(1+|\zeta_{n}|)^{10}} \min \bigg\{ \mathrm{e}^{-(\frac{1}{h_{n}-h_{n+1}})^{\frac{10}{\eta}}},  \mathrm{e}^{-(\frac{1}{r_{n}-r_{n+1}})^{\frac{10}{\eta}}}\bigg\}.
	\end{equation*}
In fact, by Proposition \ref{step} there exist $B_{n}\in C^{\omega}(2\mathbb{T}^{d}_{h_{n+1}}, \mathrm{SL}(2,\mathbb{C}))$, $F_{n+1}\in \mathcal{B}_{h_{n+1},r_{n+1}}$, $A_{n+1}\in \mathcal{M}$ such that
\begin{equation*}
	B_{n}(x+\alpha)^{-1} A_{n}\mathrm{e}^{F_{n}(x)} B_{n}(x) = A_{n+1}\mathrm{e}^{F_{n+1}(x)}.
\end{equation*}
Let $\Phi_{n+1}=\Phi_{n} \cdot B_{n}$. Then 
\begin{equation*}
	\Phi_{n+1}(x+\alpha)^{-1} A{\rm e}^{F(x)}\Phi_{n+1}(x) = A_{n+1}{\rm e}^{F_{n+1}(x)}.
\end{equation*}
To obtain the estimates of $\Phi_{n+1}$, $F_{n+1}$, $A_{n+1}$,  we distinguish two cases:

{\bf Non-resonant case:} If $\rho\notin \Theta_{n}$,  which means $A_{n}\in \mathcal{NR}(N_{n}, \varepsilon_{n}^{\frac{1}{10}})$.
Then by Proposition \ref{step}, we have  $B_{n}=\mathrm{e}^{Y_{n}}$ with estimates
\begin{equation*}
	\|Y_{n}\|_{h_{n}} \leq \varepsilon_{n}^{\frac{1}{2}}, \quad \|F_{n+1}\|_{h_{n+1}} \leq 2\varepsilon_{n}^{3}= \varepsilon_{n+1}, \quad A_{n+1}=A_{n}.
\end{equation*}
Hence $\rho_{n+1}=\rho_{n}\in \mathbb{R}$ and $|\zeta_{n+1}|=|\zeta_{n}|\leq |\zeta|$. It is obvious that 
\begin{equation*}
	\|\Phi_{n+1}\|_{0} = \|\Phi_{n}\cdot B_{n}\|_{0} \leq {\rm e}^{C_{\eta}|\log \varepsilon_{n+1}|^{\frac{2}{2+\eta}}}.
\end{equation*}
	This proves Proposition \ref{elliparabolic}\ref{item:nonres}.

{\bf Resonant case:} If $\rho\in \Theta_{n}(\mathbf{k}_{n}^{*})$, which means $A_{n}\in \mathcal{RS}(N_{n}, \varepsilon_{n}^{\frac{1}{10}})$.
Then by Proposition \ref{step}, we have following estimates
\begin{equation*}
  \|B_{n}\|_{0} \leq {\rm e}^{|\log \varepsilon_{n}|^{\frac{2}{2+\eta}}}, \quad \|F_{n+1}\|_{h_{n+1}} \leq \varepsilon_{n}^{100}<\varepsilon_{n+1}.
\end{equation*}
Moreover, $A_{n+1}$ takes the form
\begin{equation*}
A_{n+1}=\begin{pmatrix}
		\mathrm{e}^{\mathrm{i}\rho_{n+1}} & \zeta_{n+1}\\
		0&\mathrm{e}^{-\mathrm{i}\rho_{n+1}}
	\end{pmatrix} \ \ \text{or} \ \ A_{n+1}=\begin{pmatrix}
		\mathrm{e}^{\mathrm{i}\rho_{n+1}} &0 \\
		\zeta_{n+1}&\mathrm{e}^{-\mathrm{i}\rho_{n+1}}
	\end{pmatrix},
\end{equation*}
where $\rho_{n+1}=\rho_{n}-\frac{\langle \mathbf{k},\alpha\rangle}{2}\in \mathbb{R}$ with $|\rho_{n+1}|\leq \varepsilon_{n}^{\frac{1}{10}}$ and $|\zeta_{n+1}|\leq \varepsilon_{n}^{\frac{9}{10}}$. This proves Proposition \ref{elliparabolic}\ref{item:reson}. It is easy to see that
\begin{equation*}
	\|\Phi_{n+1}\|_{0}=\|\Phi_{n}\cdot B_{n}\|_{0} \leq {\rm e}^{C_{\eta}|\log \varepsilon_{n}|^{\frac{2}{2+\eta}}}  {\rm e}^{|\log \varepsilon_{n}|^{\frac{2}{2+\eta}}} \leq {\rm e}^{C_{\eta}|\log \varepsilon_{n+1}|^{\frac{2}{2+\eta}}}.
\end{equation*}

When $\rho\in \Theta_{n_{j}}(\mathbf{k}^{*}_{n_{j}})\cap \Theta_{n_{j+1}}(\mathbf{k}^{*}_{n_{j+1}})$,  on the one hand, it follows from $$\|2\rho_{n_{j+1}}-\langle {\bf k}^{*}_{n_{j+1}},\alpha\rangle\|_{\mathbb{T}}\leq \varepsilon_{n_{j+1}}^{\frac{1}{10}}$$
and  Lemma  \ref{dioestimate} that
\begin{equation*}
	2|\rho_{n_{j+1}}| \geq \gamma (1+|\mathbf{k}^{*}_{n_{j+1}}|_{\eta})^{-C_{1} |\mathbf{k}^{*}_{n_{j+1}}|_{\eta}^{\frac{1}{\eta+1}}} - \varepsilon_{n_{j+1}}^{\frac{1}{10}} \geq \mathrm{e}^{-C_{3}|\mathbf{k}_{n_{j+1}}|_{\eta}^{\frac{2}{2+\eta}} }.
\end{equation*}
  On the other hand, there is no resonance between $n_j$-th step and $n_{j+1}$-th step, and according to Proposition \ref{step}, we have $\rho_{1+n_{j}} = \rho_{n_{j+1}}$ and $|\rho_{1+n_{j}}| \leq \varepsilon_{n_{j}}^{\frac{1}{10}}$. To sum up, we obtain that
\begin{equation*}
	\frac{1}{2}\exp (-C_{3} |{\bf k}_{n_{j+1}}^{*}|_{\eta}^{\frac{2}{\eta+2}}) \leq \varepsilon_{n_{j}}^{\frac{1}{10}} \leq \exp (-\frac{1}{10}|{\bf k}^{*}_{n_{j}}|_{\eta}^{\frac{8}{8+\eta}} ),
\end{equation*}
where the second inequality uses $|{\bf k}^{*}_{n_{j}}|_{\eta}\leq N_{n_{j}}\leq |\log \varepsilon_{n_{j}}|^{1+\frac{\eta}{8}}$, which  shows that 
\begin{equation*}
	|{\bf k}^{*}_{n_{j+1}}|_{\eta} \geq |{\bf k}^{*}_{n_{j}}|_{\eta}^{1+\frac{\eta}{4+\eta}}.
\end{equation*}
Hence we finish the whole proof.
\end{proof}

\subsubsection{Reducibility of almost-periodic cocycle}

The following Corollary \ref{finite} shows that $(\alpha,A\mathrm{e}^{F})$ is reducible provided that $\rho$ belongs to at most finitely many sets $\Theta_{n}$. Let $\bar{\Theta}= \limsup_{n\rightarrow \infty}\Theta_{n}$.
\begin{corollary}\label{finite}
 If $\rho\notin \bar{\Theta}$, then there exists  $\Psi'\in C^{\omega}(2\mathbb{T}^{d}, {\rm SL}(2,\mathbb{C}))$ such that
	\begin{equation*}
		\Psi'(x+\alpha)^{-1}A{\rm e}^{F(x)}\Psi'(x) =A'.
	\end{equation*}
Indeed, let $\tilde{n}$ such that $\rho\notin \Theta_{n}$ for any $n\geq \tilde{n}$. Then $A'$ takes the precise form:
\begin{enumerate}[font=\normalfont, label={(\arabic*)}]
	\item \label{item:elli} If $\rho_{\tilde{n}}\neq 0$, then  $A'=\begin{pmatrix}
		\mathrm{e}^{\mathrm{i}\rho_{\tilde{n}}}&0\\
		0&\mathrm{e}^{-\mathrm{i}\rho_{\tilde{n}}}
	\end{pmatrix}$.

\item \label{item:para} If $\rho_{\tilde{n}}=0,$  then $A'=\begin{pmatrix}
	1&\zeta_{\tilde{n}}\\0&1
\end{pmatrix}$.
\end{enumerate}
\end{corollary}

\begin{proof}
	By Proposition \ref{elliparabolic}, there exist $\Phi_{\tilde{n}}$, $F_{\tilde{n}}$, $A_{\tilde{n}}$ such that
	\begin{equation*}
		\Phi_{\tilde{n}}(x+\alpha)^{-1} A{\rm e}^{F(x)} \Phi_{\tilde{n}}(x) = A_{\tilde{n}} {\rm e}^{F_{\tilde{n}}(x)}.
	\end{equation*}
 Since no resonance occurs for any $n\geq \tilde{n}$ by the definition of $\rho$, we use Proposition \ref{elliparabolic}\ref{item:nonres} iteratively to obtain $Y_{n}$ and $F_{n}$ for $n\geq \tilde{n}$ such that 
	\begin{equation*}
		\mathrm{e}^{-Y_{n}(x+\alpha)} A_{\tilde{n}}\mathrm{e}^{F_{n}(x)} \mathrm{e}^{Y_{n}(x)} = A_{\tilde{n}} \mathrm{e}^{F_{n+1}(x)},
	\end{equation*}
	with $\|Y_{n}\|_{h_{n}} \leq \varepsilon_{n}^{\frac{1}{2}}$ and $\|F_{n}\|_{h_{n}} \leq \varepsilon_{n}$. 

If $\rho_{\tilde{n}}\neq 0$, then there exists $P\in \mathcal{M}$ such that 
\begin{equation*}
	P^{-1} A_{\tilde{n}}P = \begin{pmatrix}
		\mathrm{e}^{\mathrm{i}\rho_{\tilde{n}}}& 0\\
		0&\mathrm{e}^{-\mathrm{i}\rho_{\tilde{n}}}
	\end{pmatrix}=:A'.
\end{equation*}
We let $\Psi' = \Phi_{\tilde{n}} \cdot \prod_{n=\tilde{n}}^{\infty}{\rm e}^{Y_{n}} \cdot P$. 

If  $\rho_{\tilde{n}}= 0$, then we let $\Psi'=\Phi_{\tilde{n}} \cdot \prod_{n=\tilde{n}}^{\infty}{\rm e}^{Y_{n}}$ for the case $A_{\tilde{n}}=\begin{pmatrix}
	1&\zeta_{\tilde{n}}\\0&1
\end{pmatrix}$. Otherwise we choose $H=\begin{pmatrix}
	0&\mathrm{i}\\ \mathrm{i}&0
\end{pmatrix}$ so that
\begin{equation*}
	H^{-1} \begin{pmatrix}
		1&0\\\zeta_{\tilde{n}}&1
	\end{pmatrix} H= \begin{pmatrix}
	1&\zeta_{\tilde{n}}\\0&1
\end{pmatrix}=:A',
\end{equation*}
which finishes the proof by letting $\Psi'=\Phi_{\tilde{n}} \cdot \prod_{n=\tilde{n}}^{\infty}{\rm e}^{Y_{n}} \cdot H$.
\end{proof}

\subsubsection{Growth of the cocycles:}

Corollary \ref{finite} shows that  the cocycle is reducible if $\rho\notin \bar{\Theta}$. In the following, we will show the cocycle has sublinear growth if $\rho\in \bar{\Theta}$.

\begin{corollary}\label{upper}
	If $\rho\in \bar{\Theta}$, then 
		\begin{equation*}
		\|\mathcal{A}_{j}\|_{0}\leq o(1+j).
	\end{equation*}
where $(j\alpha,\mathcal{A}_{j}(x) ):=(\alpha, A\mathrm{e}^{F(x)})^j$. 
\end{corollary}

\begin{proof}
To control the growth of the cocycles, we need the following
\begin{lemma}[\cite{AFK, ZW}]\label{afk}
We have that
\begin{equation*}
	M_l(\id+y_l)\cdots M_0(\id+y_0)=M^{(l)}(\id+y^{(l)}),
\end{equation*}
where $M^{(l)}=M_l\cdots M_0$ and
\begin{equation*}
	\| y^{(l)}\| \leq \mathrm{e}^{\sum_{k=0}^{l}\| M^{(k)}\|^2\| y_k\|}-1.
\end{equation*}
\end{lemma}

By Proposition \ref{elliparabolic}, $(\alpha, A\mathrm{e}^{F(x)})$ is almost reducible. Thus, we have 
\begin{equation*}
\mathcal{A}_{j}(x) =	\Phi_{n}(x+j\alpha) \Big(\prod_{s=j-1}^{0} A_{n}{\rm e}^{F_{n}(x+s\alpha)} \Big) \Phi_{n}(x)^{-1}.
\end{equation*}
Then by Lemma \ref{afk} and $\|A_{n}^{j}\| \leq 1+j|\zeta_{n}|$, it follows that
	\begin{equation*}
		\begin{split}
			\|\mathcal{A}_{j}\|_{0}
			&\leq \|\Phi_{n}\|_{0}^{2} \cdot \|A_{n}\| \cdot \|A_{n}^{j-1}\| \cdot{\rm e}^{\|F_{n}\|_{0} \|A_{n}\| \sum_{l=1}^{j} (1+|\zeta_{n}|(j-l))}\\
			&\leq (1+2|\zeta|)\cdot (1+j |\zeta_{n}|)\cdot\|\Phi_{n}\|_{0}^{2} \cdot{\rm e}^{10\varepsilon_{n} (j+j^{2}|\zeta_{n}|)}.
		\end{split}
	\end{equation*}
	For any $j\in \mathbb{N}$, one can construct an interval $\mathbb{I}_{n}$ such that 
	\begin{equation*}
		j \in \mathbb{I}_{n}:= (\varepsilon_{n}^{-\frac{1}{8}}, \varepsilon_{n}^{-\frac{1}{2}}).
	\end{equation*}
	Since $\mathbb{I}_{n}\cap \mathbb{I}_{n+1}\neq \varnothing$, we conclude that $\{\mathbb{I}_{n}\}_{n\in\mathbb{N}}$ cover all the $j$ tending to $\infty$, and
	\begin{equation*}
		\|\mathcal{A}_{j}\|_{0} \leq 2j(1+2|\zeta|) \cdot \|\Phi_{n}\|_{0}^{2} |\zeta_{n}| .
	\end{equation*}
Note that if $\rho\in \Theta_{n}$, by Proposition \ref{elliparabolic}\ref{item:reson} we have
\begin{equation*}
	\|\Phi_{n+1}\|_{0}^{2} \cdot |\zeta_{n+1}|  \leq \varepsilon_{n+1}^{\frac{1}{4}},
\end{equation*} 
then the result follows from the assumption.
\end{proof}

\subsection{Hyperbolic case}

Recall that
\begin{equation*}
	\mathcal{M} := \bigg\{\begin{pmatrix}
		{\rm e}^{{\rm i}\xi}&\zeta\\
		0&{\rm e}^{-{\rm i}\xi}
	\end{pmatrix}:\xi,\zeta \in \mathbb{C}\bigg\} \cup \bigg\{\begin{pmatrix}
		{\rm e}^{{\rm i}\xi}&0\\
		\zeta&{\rm e}^{-{\rm i}\xi}
	\end{pmatrix}:\xi,\zeta \in \mathbb{C} \bigg\}.
\end{equation*}
To obtain the reducibility result for hyperbolic $A\in\mathcal{M}$, first we need the following simple observation:

\begin{lemma}\label{diaghyper}
	Let $d\in \mathbb{N}^{+} \cup \{\infty\}$, $\eta>0$, $h>0$, $r\in (0,1]$, $\gamma>0$, $\tau>1$, $\alpha\in {\rm DC}^{d}_{\gamma,\tau}$. Suppose that $A\in \mathcal{M}$  with $\mathrm{Im}\xi\neq 0$ and  $\zeta=0$, $F\in \mathcal{B}_{h,r}$ with 
	\begin{equation}\label{hyass}
		\|F\|_{h} < \varepsilon< 
		\min \{ 10^{-8},  |\mathrm{Im}\xi|^{3} \}, 
	\end{equation} 
	then $(\alpha, A\mathrm{e}^{F})$ is reducible to $(\alpha,A)$.
\end{lemma}

\begin{proof}

Let $\sigma=\varepsilon^{\frac{1}{3}}$ and 
\begin{equation*}
	\begin{split}
		&\Lambda_{1}=\{{\bf k}\in \Gamma_{r}: |{\rm e}^{{\rm i}\langle {\bf k},\alpha \rangle}-1|\geq \sigma\},\\
		&\Lambda_{2}=\{{\bf k}\in \Gamma_{r}: |{\rm e}^{{\rm i}(\langle {\bf k},\alpha \rangle\pm 2\xi)}-1|\geq \sigma\}.
	\end{split}
\end{equation*}
Then we define the decomposition $\mathcal{B}_{h,r} = \mathcal{B}_{h,r}^{\rm nre}(\sigma) \oplus \mathcal{B}_{h,r}^{\rm re}(\sigma)$ with respect to $A$, $\sigma$, where $\mathcal{B}_{h,r}^{\rm nre}(\sigma)$ is defined to be the space of all $F\in \mathcal{B}_{h,r}$ of the form
	\begin{equation}\label{nrspace}
		F(x)=\sum_{{\bf k}\in \Lambda_{1}} \begin{pmatrix}
			a_{\bf k}&0\\0&-a_{\bf k}
		\end{pmatrix}{\rm e}^{{\rm i}\langle{\bf k},x\rangle} + \sum_{{\bf k}\in \Lambda_{2}} \begin{pmatrix}
			0&b_{\bf k}\\c_{\bf k}&0
		\end{pmatrix}{\rm e}^{{\rm i}\langle{\bf k},x\rangle},
	\end{equation}
	and $\mathcal{B}_{h,r}^{\rm re}(\sigma)$ is defined to be the space of all $F\in \mathcal{B}_{h,r}$ of the form
	\begin{equation}\label{respace}
		F(x)=\sum_{{\bf k}\in \Gamma_{r}\backslash \Lambda_{1}} \begin{pmatrix}
			a_{\bf k}&0\\0&-a_{\bf k}
		\end{pmatrix}{\rm e}^{{\rm i}\langle{\bf k},x\rangle} + \sum_{{\bf k}\in \Gamma_{r}\backslash \Lambda_{2}} \begin{pmatrix}
			0&b_{\bf k}\\c_{\bf k}&0
		\end{pmatrix}{\rm e}^{{\rm i}\langle{\bf k},x\rangle}.
	\end{equation}
	
 For any $Y\in \mathcal{B}_{h,r}^{\rm nre}(\sigma)$, we have
	\begin{equation*}
		\begin{split}
			(L_{A}Y)(x)&= \sum_{{\bf k}\in \Lambda_{1}}\begin{pmatrix}
				a_{\bf k} ({\rm e}^{ {\rm i}\langle {\bf k},\alpha\rangle}-1)&0\\0&-a_{\bf k}({\rm e}^{{\rm i}\langle {\bf k},\alpha\rangle}-1)
			\end{pmatrix} {\rm e}^{{\rm i}\langle{\bf k},x\rangle} \\
			&\quad +\sum_{{\bf k}\in \Lambda_{2}}
			\begin{pmatrix}
				0&b_{\bf k}({\rm e}^{{\rm i}\langle {\bf k},\alpha\rangle-2\xi}-1)\\c_{\bf k}({\rm e}^{{\rm i}\langle {\bf k},\alpha\rangle+2\xi}-1)&0
			\end{pmatrix}{\rm e}^{{\rm i}\langle{\bf k},x\rangle}. 
		\end{split}
	\end{equation*}
	Thus $L_{A}$ is invertible on $\mathcal{B}_{h,r}^{\rm nre}(\sigma)$ and $\|L_{A}^{-1}\|\leq \frac{1}{\sigma}$, which means the decomposition for \eqref{nrspace} and \eqref{respace} is well-defined. 
	
Just note by the assumption \eqref{hyass}, we have
$$ |{\rm e}^{{\rm i}(\langle {\bf k},\alpha \rangle\pm 2\xi)}-1|\geq  2|{\rm Im} \xi| \geq \sigma, \qquad \forall \ \mathbf{k}\in \mathbb{Z}^{d}_{*},$$
which implies $\Gamma_{r}\backslash\Lambda_{2}=\varnothing$. Thus  by Lemma \ref{hy}, there exist $Y\in  \mathcal{B}_{h,r}^{\rm nre}(\sigma)$ and $F^{\rm re}\in \mathcal{B}_{h,r}^{\rm re}(\sigma)$ such that
	\begin{equation*}
		{\rm e}^{-Y(x+\alpha)}A{\rm e}^{F(x)} {\rm e}^{Y(x)} = A{\rm e}^{F'^{\rm re}(x)}=: \begin{pmatrix}
			{\rm e}^{\mathrm{i}\xi}{\rm e}^{f(x)}&0\\
			0&{\rm e}^{-\mathrm{i}\xi}{\rm e}^{-f(x)} 
		\end{pmatrix}.
	\end{equation*}
	
Since $\alpha\in {\rm DC}^{d}_{\gamma,\tau}$, and $\widehat{f}_{\bf 0}=0$  by  $f\in \mathcal{B}_{h,r}[\mathbb{C}],$ then 
	\begin{equation*}
		\varphi(x+\alpha)-\varphi(x)=f(x), \quad f\in \mathcal{B}_{h,r}[\mathbb{C}],
	\end{equation*} 	
	always has a solution $\varphi\in C^{\omega}(\mathbb{T}^{d}_{h'},\C)$ with $h'\in (0,h)$. Let 
$\Psi= \mathrm{e}^{Y} \cdot \begin{pmatrix}
		{\rm e}^{\varphi(x)}&0\\0&{\rm e}^{-\varphi(x)}
	\end{pmatrix}\in C^{\omega}(\mathbb{T}^{d}_{h'},{\rm SL}(2,\mathbb{C}))$. It follows that
	\begin{equation*}
		\Psi(x+\alpha)^{-1}A{\rm e}^{F(x)}\Psi(x)=A.
	\end{equation*}
The proof is finished.
\end{proof}

As a consequence, we have the following:

\begin{proposition}\label{hyperbolic}
Let  $d\in \mathbb{N}^{+} \cup \{\infty\}$, $\eta>0$, $h>0$, $h'\in (0,h)$, $r\in (0,1]$,   $r'\in (0,r)$, $\gamma>0$, $\tau>1$, $\alpha\in {\rm DC}^{d}_{\gamma,\tau}$. Suppose that  $A=\begin{pmatrix}
	{\rm e}^{\mathrm{i}\xi} &\zeta\\
	0&{\rm e}^{-\mathrm{i}\xi} 
\end{pmatrix}$ with $\mathrm{Im}\xi \neq 0$ and $F\in \mathcal{B}_{h,r}$. 
	There exist $\varepsilon=\varepsilon(\eta,h,h',r,r', \gamma,\tau,|\zeta|)$ and $c=c(\eta,\gamma,\tau)$ such that if
	\begin{equation*}
		\|F\|_{h} < \varepsilon< \frac{c}{(1+|\zeta|)^{10}} \min \bigg\{{\rm e}^{-(\frac{1}{h-h'})^{\frac{10}{\eta}}}, {\rm e}^{-(\frac{1}{r-r'})^{\frac{10}{\eta}}} \bigg\},
	\end{equation*} 
then $(\alpha,A{\rm e}^{F})$ is reducible to $(\alpha, A')$,  where $A'=\begin{pmatrix}
	\mathrm{e}^{\mathrm{i}\xi'} &0\\
	0&\mathrm{e}^{-\mathrm{i}\xi'}
\end{pmatrix}$ with $\mathrm{Im}\xi' = \mathrm{Im}\xi$.
\end{proposition}

\begin{proof}
We distinguish the proof into two cases:\\

\textbf{Case 1: Strong hyperbolic case.} If $|\mathrm{Im}\xi| \geq \frac{1}{2}$, then there exists $P\in  \mathcal{M}$ with $\|P-\mathrm{id}\|\leq 2|\zeta|$ such that 
	\begin{equation*}
		P^{-1} A\mathrm{e}^{F(x)} P = \begin{pmatrix}
			\mathrm{e}^{\mathrm{i}\xi} &0\\
			0&\mathrm{e}^{-\mathrm{i}\xi}
		\end{pmatrix}\mathrm{e}^{P^{-1}F(x) P} =:A'\mathrm{e}^{F'(x)},
	\end{equation*}
where $F'\in \mathcal{B}_{h,r}$ with $\|F'\|_{h} \leq \varepsilon^{\frac{9}{10}}$. By Lemma \ref{diaghyper}, $(\alpha,A'\mathrm{e}^{F'})$ is reducible to $(\alpha, A')$ with $\xi'=\xi$.\\

\textbf{Case 2: Weak hyperbolic case.}  If $|\mathrm{Im}\xi| < \frac{1}{2}$, by Proposition \ref{step} there exist $B_{n}\in C^{\omega}(2\mathbb{T}^{d}_{h_{n}}, \mathrm{SL}(2,\mathbb{C}))$, $F_{n}\in \mathcal{B}_{h_{n},r_{n}}$ and $A_{n}\in \mathcal{M}$ such that
\begin{equation*}
	B_{n}(x+\alpha)^{-1} A_{n}\mathrm{e}^{F_{n}(x)} B_{n}(x) = A_{n+1}\mathrm{e}^{F_{n+1}(x)},
\end{equation*}
where $A_{n}=\begin{pmatrix}
	\mathrm{e}^{\mathrm{i}\xi_{n}} & \zeta_{n}\\
	0&\mathrm{e}^{-\mathrm{i}\xi_{n}}
\end{pmatrix} \ \text{or} \ A_{n}=\begin{pmatrix}
	\mathrm{e}^{\mathrm{i}\xi_{n}} &0 \\
	\zeta_{n}&\mathrm{e}^{-\mathrm{i}\xi_{n}}
\end{pmatrix}$ with $\mathrm{Im}\xi_{n} =\mathrm{Im}\xi$ and $|\zeta_{n}|\leq |\zeta|$. According to the selection of \eqref{star}, the iteration is ensured by
\begin{equation*}
	\varepsilon_{n} < \frac{c}{(1+|\zeta_{n}|)^{10}} \min \bigg\{ \mathrm{e}^{-(\frac{1}{h_{n}-h_{n+1}})^{\frac{10}{\eta}}},  \mathrm{e}^{-(\frac{1}{r_{n}-r_{n+1}})^{\frac{10}{\eta}}}\bigg\}.
\end{equation*}

Let $\Phi_{0}=\mathrm{id}$ and $\Phi_{n}= \Phi_{n-1} \cdot B_{n-1}$. Then for $n\geq 0$,
\begin{equation*}
	\Phi_{n}(x+\alpha)^{-1} A\mathrm{e}^{F(x)} \Phi_{n}(x) = A_{n}\mathrm{e}^{F_{n}(x)},
\end{equation*}
with $\|F_{n}\|_{h_{n}} \leq \varepsilon_{n}$. Furthermore, there exists $P_{n}\in \mathcal{M}$ with $\|P_{n}\|\leq {\rm e}^{|\log \varepsilon_{n}|^{\frac{2}{2+\eta}}}$ such that 
\begin{equation*}
	P_{n}^{-1} A_{n}\mathrm{e}^{F_{n}(x)} P_{n} = \begin{pmatrix}
		\mathrm{e}^{\mathrm{i}\xi_{n}} &0\\
		0&\mathrm{e}^{-\mathrm{i}\xi_{n}}
	\end{pmatrix}  \mathrm{e}^{P_{n}^{-1}F_{n}(x) P_{n}} =:A_{n}'\mathrm{e}^{F_{n}'(x)},
\end{equation*}
with $\|F_{n}'\|_{h_{n}} \leq \varepsilon_{n}^{\frac{9}{10}}$. 
Since $\mathrm{Im}\xi_{n} =\mathrm{Im}\xi$, let us choose the smallest $\tilde{n}$ such that 
\begin{equation*}
\varepsilon_{\tilde{n}}^{\frac{9}{10}} \leq \min \{10^{-8}, |\mathrm{Im}\xi_{\tilde{n}}|^{3} \}.
\end{equation*}
By Lemma \ref{diaghyper}, $(\alpha, A_{\tilde{n}}'\mathrm{e}^{F_{\tilde{n}}'(x)})$ is reducible to $(\alpha, A_{\tilde{n}}')$. Denote $A' = A_{\tilde{n}}'$ and $\xi'=\xi_{\tilde{n}}$. This finishes the proof.
\end{proof}

\section{Applications in Schr\"odinger operators}

In this section, we give the applications for Schr\"odinger operators.
Let us rewrite the Schr\"odinger cocycle $S_{E,\lambda v}(x)=\begin{pmatrix}
		E-\lambda v(x)&-1\\1&0
	\end{pmatrix} = A_{E}{\rm e}^{F_{v}(x)}$, where
	\begin{equation*}
		A_{E}=\begin{pmatrix}
				E&-1\\1&0
			\end{pmatrix}, \quad F_{v}(x)=\begin{pmatrix}
				0&0\\ \lambda v(x)&0
			\end{pmatrix}.
	\end{equation*}
	Since $A_{E}\in {\rm SL}(2,\mathbb{C})$ one knows that the eigenvalues of $A_{E}$ are $  \frac{E}{2} \pm\frac{\sqrt{E^{2}-4}}{2}$.

\subsection{Proof of Theorem \ref{maintheorem1} and Theorem \ref{maintheorem3}:}

Note that one can always conjugate $A_{E}$ to the upper triangular matrix $A$  whose upper-right term is $\zeta$. To apply Proposition \ref{elliparabolic} and Proposition \ref{hyperbolic} for all $E\in \mathbb{C}$, and to obtain uniform smallness condition of $|\lambda|$, the key observation is that $|\zeta|$ is uniformly bounded with respect to $E$.

\textbf{Case 1: $E\in [-2,2]$.} The eigenvalues of $A_{E}$ are ${\rm e}^{\pm{\rm i}\rho}$ with $\rho\in \mathbb{R}$. Let $	U_{0}=\frac{1}{\sqrt{2}} \begin{pmatrix}
			{\rm e}^{{\rm i}\rho}&-1\\1&{\rm e}^{-{\rm i}{\rho}}
		\end{pmatrix}$, then we have
	\begin{equation*}
		U_{0}^{-1}A_{E}U_{0} = \begin{pmatrix}
			{\rm e}^{{\rm i}\rho}&\zeta\\0&{\rm e}^{-{\rm i}\rho}
		\end{pmatrix} =:A\in \mathcal{M},
	\end{equation*}
where $\zeta=-1-{\rm e}^{-2{\rm i}\rho}$ and thus $|\zeta|\leq 2$. Let $F= U_{0}^{-1}\cdot F_{v}\cdot U_{0}\in \mathcal{B}_{h,r}$. Then $\|F\|_{h}= \lambda \|v\|_{h}$. If we choose $\lambda_{0}$ such that
\begin{equation}\label{epsilon}
	\lambda_{0}\|v\|_{h} < \frac{c}{3^{10}} \min \bigg\{{\rm e}^{-(\frac{2}{h})^{\frac{10}{\eta}}}, {\rm e}^{-(\frac{2}{r})^{\frac{10}{\eta}}} \bigg\},
\end{equation}
then one can apply  Proposition \ref{elliparabolic} to show $(\alpha,S_{E,\lambda v})$ is almost reducible, consequently 
one can use Corollary \ref{finite} and Corollary \ref{upper} to obtain that
	\begin{equation*}
		L(E)=L(\alpha,A{\rm e}^{F})=0=\log \Big|\frac{E}{2}+\frac{\sqrt{E^{2}-4}}{2}\Big|,
	\end{equation*}
Therefore, $(\alpha, S_{E,\lambda v})\notin \mathcal{UH}$, and by Proposition \ref{equi} we have $E\in \Sigma_{\lambda v,\alpha}$.\\

\textbf{Case 2:  $E\in \mathbb{C}\backslash [-2,2]$.}  The eigenvalues of $A_{E}$ are ${\rm e}^{\pm \mathrm{i}\xi}$ with ${\rm Im}\xi<0$. We can choose $U_{0}=\frac{1}{\sqrt{|{\rm e}^{2\mathrm{i}\xi}|+1}}\begin{pmatrix}
	{\rm e}^{\mathrm{i}\xi}&-1\\
	1&{\rm e}^{\mathrm{i}\xi}
\end{pmatrix}$ and $\zeta=-1-\mathrm{e}^{-2\mathrm{i}\xi}$  so that 
\begin{equation*}
	U_{0}^{-1}A_{E}U_{0} = \begin{pmatrix}
		{\rm e}^{\mathrm{i}\xi}&\zeta\\0&{\rm e}^{-\mathrm{i}\xi}
	\end{pmatrix}=: A\in \mathcal{M}.
\end{equation*}
Let $F = U_{0}^{-1}\cdot F_{v} \cdot U_{0}$. By $|\zeta|\leq 2$, one can also choose $\lambda _{0}$ satisfying \eqref{epsilon}. It follows from Proposition \ref{hyperbolic} that
the cocycle $(\alpha,S_{E,\lambda v})$ is reducible to some hyperbolic matrix $A'\in {\rm SL}(2,\mathbb{C})$, with
\begin{equation*}
	L(E)=L(\alpha,A_{0}{\rm e}^{F_{0}})=|{\rm Im}\xi|=\log \Big|\frac{E}{2}+\frac{\sqrt{E^{2}-4}}{2}\Big|.
\end{equation*}
Therefore,  $(\alpha,S_{E,\lambda v})\in \mathcal{UH}$, and by Proposition \ref{equi} we have $E\notin \Sigma_{\lambda v,\alpha}$.\qed

\subsection{Proof of Theorem \ref{maintheorem4}:}
Suppose that \eqref{epsilon} holds, then by Proposition \ref{elliparabolic} and Corollary \ref{finite}, it is enough to show that $\rho\notin \bar{\Theta}$ if $2\rho=\langle {\bf k},\alpha\rangle \mod 2\pi$ or $\rho \in{\rm DC}_{\kappa,\tau}(\alpha)$. We show $\rho\notin \bar{\Theta}$ by contradiction. In fact, if $\rho\in \bar{\Theta}$, then we label the resonant steps $\{n_{j}\}\subset \mathbb{N}$  such that
\begin{equation}\label{infityresonant}
	\|2\rho_{n_{j}}-\langle {\bf k}^{*}_{n_{j}},\alpha \rangle\|_{\mathbb{T}} \leq \varepsilon_{n_{j}}^{\frac{1}{10}}, \quad {\bf k}^{*}_{n_{j}}\in \mathcal{T}_{N_{n_{j}}}\Gamma_{r_{n_{j}}}.
\end{equation}
Let ${\bf d}_j = \sum_{s=1}^{j} {\bf k}^{*}_{n_{s}}$ for each $j\in \mathbb{N}$. By Proposition \ref{elliparabolic}\ref{item:reson} and Lemma \ref{dioestimate}, for sufficiently large $n_{j}$, 
\begin{equation}\label{semidio}
	\begin{split}
		\|2\rho_{n_{j}}-\langle {\bf k}^{*}_{n_{j}},\alpha \rangle\|_{\mathbb{T}} &= \|2\rho -\langle {\bf d}_j,\alpha \rangle\|_{\mathbb{T}}\\
		& \geq \min\{\kappa,\gamma\} (1+|2{\bf d}_j|_{\eta})^{-C_{1}|2{\bf d}_j|_{\eta}^{\frac{1}{\eta+1}}}.
	\end{split}	
\end{equation}
By Proposition \ref{elliparabolic}\ref{item:spe},
\begin{equation*}
 |{\bf d}_j|_{\eta}  \leq ( 1+ 2|\log \varepsilon_{n_{j-1}}|^{-\frac{\eta}{4+\eta}} )|{\bf k}^{*}_{n_{j}}|_{\eta} \leq 2|\mathbf{k}^{*}_{n_{j}}|_{\eta}.
\end{equation*}
Combine the above inequality with \eqref{semidio}, one has
\begin{equation*}
		\|2\rho_{n_{j}}-\langle {\bf k}^{*}_{n_{j}},\alpha \rangle\|_{\mathbb{T}}  \geq C_{4}(\eta,\kappa,\gamma,h,h') {\rm e}^{-(2|\log \varepsilon_{n_{j}}|)^{\frac{2}{\eta+2}}},
\end{equation*}
which contradicts to \eqref{infityresonant}. Let us choose $\tilde{n}$ such that $\rho\notin \Theta_{\tilde{n}}$ for any $n\geq \tilde{n}$. 

If $\rho\in \mathrm{DC}(\alpha)$, we have $\rho_{\tilde{n}}\neq 0$ and by Corollary \ref{finite}\ref{item:elli},
\begin{equation*}
	\Psi'(x+\alpha)^{-1} A\mathrm{e}^{F(x)}\Psi'(x) = \begin{pmatrix}
		\mathrm{e}^{\mathrm{i}\rho_{\tilde{n}}}&0\\
		0&\mathrm{e}^{-\mathrm{i}\rho_{\tilde{n}}}
	\end{pmatrix}.
\end{equation*}
Let $\{n_{s}\}_{s=1}^{J^{*}}$ be the all resonant steps with $J^{*}<\infty$. Denote $\mathbf{m}=\sum_{s=1}^{J^{*}} \mathbf{k}^{*}_{n_{s}}$ and let $Q(x)= R_{\frac{\langle \mathbf{m},x\rangle}{2}}$. Then
\begin{equation*}
	Q(x+\alpha)^{-1}\begin{pmatrix}
		\mathrm{e}^{\mathrm{i}\rho_{\tilde{n}}}&0\\
		0&\mathrm{e}^{-\mathrm{i}\rho_{\tilde{n}}}
	\end{pmatrix} Q(x)=\begin{pmatrix}
		\mathrm{e}^{\mathrm{i}\rho}&0\\
		0&\mathrm{e}^{-\mathrm{i}\rho}
	\end{pmatrix}.
\end{equation*}
Let $\Psi:=\Psi'\cdot Q$. This finishes the proof.

If $2\rho=\langle {\bf k},\alpha\rangle \mod 2\pi$, the proof follows from Corollary \ref{finite}\ref{item:para} if $\rho_{\tilde{n}}=0$. If $\rho_{\tilde{n}}\neq0$, by using Corollary \ref{finite}\ref{item:elli}, 
\begin{equation*}
	\Psi'(x+\alpha)^{-1}A{\rm e}^{F(x)}\Psi'(x) = \begin{pmatrix}
		\mathrm{e}^{\mathrm{i}\rho_{\tilde{n}}}&0\\
		0&\mathrm{e}^{-\mathrm{i}\rho_{\tilde{n}}}
	\end{pmatrix}.
\end{equation*}
Choose $\mathbf{m}\in \mathbb{Z}^{d}_{*}$ such that $2\rho_{\tilde{n}}= \langle \mathbf{m},\alpha\rangle \mod 2\pi$ and let $Q(x)= R_{\frac{\langle \mathbf{m},x\rangle}{2}}$. Then
\begin{equation*}
	Q(x+\alpha)^{-1}\begin{pmatrix}
		\mathrm{e}^{\mathrm{i}\rho_{\tilde{n}}}&0\\
		0&\mathrm{e}^{-\mathrm{i}\rho_{\tilde{n}}}
	\end{pmatrix} Q(x)=\begin{pmatrix}
		1&0\\
		0&1
	\end{pmatrix}.
\end{equation*}
Let $\Psi :=\Psi' \cdot Q$. This finishes the proof. \qed

\subsection{Proof of Theorem \ref{maintheorem5}: }
 Let $E=2\cos 2\rho$ with $-2\rho =\langle \mathbf{m},\alpha \rangle \mod 2\pi$. Choose $|\lambda|$ sufficiently small such that
\begin{equation}\label{lambdasmall}
	|\lambda|<\min\{\lambda_{0}, \mathrm{e}^{-2h|\mathbf{m}|_{\eta}}\}.
\end{equation}  
To prove Theorem \ref{maintheorem5}, we only need to show resonant case only appears once in the setting of Proposition \ref{elliparabolic}. For simplicity, we denote $\varepsilon:= \|\lambda v\|_{h}$, and thus by \eqref{lambdasmall},
\begin{equation}\label{m}
	\varepsilon<\mathrm{e}^{-h|\mathbf{m}|_{\eta}}.
\end{equation} 
We sketch the proof into the following five steps:

{\bf Step 1.} Let $U_{0}=\frac{1}{\sqrt{2}} \begin{pmatrix}
	{\rm e}^{{\rm i}\rho}&-1\\1&{\rm e}^{-{\rm i}{\rho}}
\end{pmatrix}$. Then
\begin{equation*}
	U_{0}^{-1}A_{E}\mathrm{e}^{F_{v}(x)}U_{0} = A\mathrm{e}^{F(x)},
\end{equation*}
where $A=\begin{pmatrix}
	{\rm e}^{{\rm i}\rho}&\zeta\\0&{\rm e}^{-{\rm i}\rho}
\end{pmatrix}$ with $\zeta_{0}=-1-\mathrm{e}^{\mathrm{i}\rho}$ and 
$F = U_{0}^{-1} F_{v}  U_{0}$. We have
\begin{equation*}
	\widehat{F}_{\mathbf{m}}=U_{0}^{-1} \begin{pmatrix}
		0&0\\
		\lambda&0
	\end{pmatrix}U_{0} =\frac{\lambda}{2} \begin{pmatrix}
		\mathrm{e}^{\mathrm{i}\rho} &-1\\
		\mathrm{e}^{2\mathrm{i}\rho}&-\mathrm{e}^{\mathrm{i}\rho} 
	\end{pmatrix},
\end{equation*}
and thus $F^{2,1}(x) = \frac{\lambda}{2}  \mathrm{e}^{2\mathrm{i}\rho} \mathrm{e}^{\mathrm{i} \langle \mathbf{m},x\rangle}$ with $|(\widehat{F}_{\mathbf{m}})^{2,1}|=\frac{\varepsilon}{2}\mathrm{e}^{-h|\mathbf{m}|_{\eta}}$.

{\bf Step 2.} Let $P=\begin{pmatrix}
	1&\frac{\zeta}{\mathrm{e}^{-\mathrm{i}\rho}-\mathrm{e}^{\mathrm{i}\rho}}\\ 0 &1
\end{pmatrix}$. By \eqref{m}, we have $|\mathbf{m}|_{\eta} \leq N_{0}=\frac{2|\log {\varepsilon}|}{h-h_{1}}$ and thus $\rho\in \Theta_{0}(\mathbf{m})$. So by Lemma \ref{dioestimate}, we have $\|P\|\leq {\rm e}^{|\log \varepsilon|^{\frac{2}{2+\eta}}}$.  Direct calculation shows  that
\begin{equation*}
	P^{-1} A\mathrm{e}^{F(x)} P= A' \mathrm{e}^{F'(x)},
\end{equation*}
where $A'=\begin{pmatrix}
	\mathrm{e}^{\mathrm{i}\rho}&0\\
	0&\mathrm{e}^{-\mathrm{i}\rho}
\end{pmatrix}$ and $F' (x)= \begin{pmatrix}
	g_{1}(x)&g_{2}(x)\\F^{2,1}(x)&-g_{1}(x)
\end{pmatrix}$
with $\|g_{1}\|_{h}, \|g_{2}\|_{h}\leq\|F'\|_{h} \leq  \varepsilon^{\frac{9}{10}}=:\varepsilon'$.

{\bf Step 3.} Let $N'=C_{2} |\log\varepsilon|^{1+\frac{\eta}{2}}-N$ and $\sigma =\varepsilon'^{\frac{1}{3}}$. One can check that the space decomposition with respect to $A',\sigma$ is well-defined:
\begin{equation*}
	\begin{split}
		&\mathcal{B}^{\rm nre}_{h,r}(\sigma)= \bigg\{F(x)=\mathcal{T}_{N'} F(x)
		- \begin{pmatrix}
			0&0	\\c_{\bf m}&0
		\end{pmatrix}{\rm e}^{{\rm i}\langle m, x\rangle}  \bigg\},\\
		&\mathcal{B}^{\rm re}_{h,r}(\sigma)= \bigg\{ F(x)=\mathcal{R}_{N'} F(x)
		+\begin{pmatrix}
			0&0\\c_{\bf m}	&0
		\end{pmatrix}{\rm e}^{{\rm i}\langle m, x\rangle} \bigg\}.
	\end{split}
\end{equation*}
By Lemma \ref{hy}, there exist $Y\in\mathcal{B}_{h,r}^\mathrm{nre}(\sigma)$ and $F^\mathrm{re}\in \mathcal{B}_{h,r}^\mathrm{re}(\sigma)$ such that 
$$\mathrm{e}^{-Y(x+\alpha)} A'\mathrm{e}^{F'(x)} \mathrm{e}^{Y(x)} =A'\mathrm{e}^{F^\mathrm{re}(x)}$$ 
with $\|F^\mathrm{re}-\mathbb{P}_\mathrm{re} F'\|_{h} \leq 2\varepsilon'^{\frac{4}{3}}\leq  2\varepsilon^{\frac{6}{5}}$.

{\bf Step 4.} Let $Q(x):= R_{\frac{\langle \mathbf{m},x\rangle}{2}}$. By Lemma \ref{keep}, there exists $F_{1}\in \mathcal{B}_{h_{1},r_{1}}$ such that
\begin{equation*}
	Q(x+\alpha)^{-1} A' \mathrm{e}^{F^\mathrm{re}(x)} Q(x)= \begin{pmatrix}
		1&0\\ c_{\bf m}&1
	\end{pmatrix} \mathrm{e}^{F_{1}(x)}=:A_{1}\mathrm{e}^{F_{1}(x)}, 
\end{equation*}
where $c_{\bf m}=(\widehat{F^\mathrm{re}_\mathbf{m}})^{2,1}$ and $\|F_{1}\|_{h_{1}}\leq \varepsilon^{100}$. Hence, 
\begin{equation*}
		|c_{\bf m}-(\widehat{F}_{\mathbf{m}})^{2,1}| = |(\widehat{F^\mathrm{re}_\mathbf{m}})^{2,1}-(\mathbb{P}_\mathrm{re}\widehat{F'_\mathbf{m}})^{2,1}| \leq \|F^\mathrm{re}-\mathbb{P}_\mathrm{re}F'\|_{h} \mathrm{e}^{-h|\mathbf{m}|_{\eta}} \leq 2\varepsilon^{\frac{6}{5}}\mathrm{e}^{-h|\mathbf{m}|_{\eta}},
\end{equation*}
which shows that $|c_{\bf m}|\geq |(\widehat{F}_{\mathbf{m}})^{2,1}|- 2\varepsilon^{\frac{6}{5}}\mathrm{e}^{-h|\mathbf{m}|_{\eta}}>0$.

{\bf Step 5.}
We claim that $\rho \notin \Theta_{n}$ for any $n\geq 1$. In fact, by $\alpha\in \mathrm{DC}^{d}_{\gamma,\tau}$, one can apply inductively Proposition \ref{elliparabolic}\ref{item:nonres} to show that for $n\geq1$ and $\mathbf{k}\in \mathcal{T}_{N_{n}} \Gamma_{r_{n}}$,
\begin{equation*}
	\rho_{n}=0, \ \text{and} \ \|2\rho_{n}-\langle \mathbf{k},\alpha\rangle\|_{\mathbb{T}} =\|\langle \mathbf{k},\alpha\rangle\|_{\mathbb{T}}\geq {\rm e}^{|\log \varepsilon_{n}|^{\frac{2}{2+\eta}}}.
\end{equation*}
Then there exist $Y_{n}\in \mathcal{B}_{h_{n},r_{n}}$ and $F_{n}\in \mathcal{B}_{h_{n},r_{n}}$ such that for any $n\geq1$,
\begin{equation*}
	\mathrm{e}^{-Y_{n}(x+\alpha)}A_{1}\mathrm{e}^{F_{n}(x)}\mathrm{e}^{Y_{n}(x)}=A_{1}\mathrm{e}^{F_{n+1}(x)},
\end{equation*}
with $\|Y_{n}\|_{h_{n}} \leq \varepsilon_{n}^{\frac{1}{2}}$ and $\|F_{n}\|_{h_{n}}\leq \varepsilon_{n}$. Finally we choose $H=\begin{pmatrix}
	0&\mathrm{i}\\ \mathrm{i}&0
\end{pmatrix}$, then 
\begin{equation*}
	H^{-1}A_{1} H=\begin{pmatrix}
		1&\zeta\\0&1
	\end{pmatrix} \quad \text{with} \quad \zeta=c_{\bf m}.
\end{equation*}
Let $\Psi =U_{0} \cdot P \cdot \mathrm{e}^{Y}\cdot Q \cdot \prod_{n=1}^{\infty} \mathrm{e}^{Y_{n}}\cdot H$. The proof is finished by $c_\mathbf{m} \neq0$.\qed

\section{One-frequency examples}

In this section,  by Avila's global theory of one-frequency analytic $\mathrm{SL}(2,\mathbb{C})$ cocycles \cite{Av0}, we determine the spectrum of two examples of one-frequency non-self-adjoint Schr\"odinger operators.

\subsection{Proof of  Theorem \ref{maintheorem3-1}:}

To calculate the Lyapunov exponent of  $(\alpha, S)$ with $S(x)=\begin{pmatrix}
	E-\lambda \mathrm{e}^{\mathrm{i}x} &-1\\1&0
\end{pmatrix}$, let us complexify the phase
\begin{equation*}
	S_{\epsilon}(x):=S(x+\mathrm{i}\epsilon) =\begin{pmatrix}
		E-\lambda \mathrm{e}^{\mathrm{i}(x+\mathrm{i}\epsilon)} &-1\\1&0
	\end{pmatrix}.
\end{equation*}
Denote by $L(E,\epsilon):=L(\alpha,S_{\epsilon})$ the Lyapunov exponent of $(\alpha, S_{\epsilon})$ and by $\omega(E,\epsilon):= \omega(\alpha, S_{\epsilon})$ the acceleration of that.  For sufficiently large $\epsilon>0$,
\begin{equation*}
	S_{\epsilon}(x)= \begin{pmatrix}
		E&-1\\1&0
	\end{pmatrix}+o(1),
\end{equation*}
then by the continuity of Lyapunov exponent \cite{BJ,JK},
\begin{equation*}
	L(E,\epsilon) =\log \Big|\frac{E}{2}+\frac{\sqrt{E^{2}-4}}{2}\Big|+o(1).
\end{equation*}
According to the quantization of acceleration in Theorem \ref{acce}, for $\epsilon>0$ large enough, 
\begin{equation}\label{pac}
	\omega(E,\epsilon)=0, \quad L(E,\epsilon) =\log \Big|\frac{E}{2}+\frac{\sqrt{E^{2}-4}}{2}\Big|.
\end{equation}
The similar argument works for sufficiently small $\epsilon<0$,
\begin{equation*}
	S_{\epsilon}(x)= \mathrm{e}^{\mathrm{i}x} \mathrm{e}^{-\epsilon} \begin{pmatrix}
		-\lambda &0\\0&0
	\end{pmatrix} + o(1),
\end{equation*}
and furthermore, 
\begin{equation}\label{nac}
	\omega(E,\epsilon)=-1, \quad L(E,\epsilon) = -\epsilon+ \log |\lambda|.
\end{equation}

Abbreviate the spectrum $\Sigma_{\lambda \exp,\alpha}$ by $\Sigma$. Let us calculate $L(E)$ for $E\in \Sigma$ firstly. For $|\lambda|\leq 1$, by the convexity of $L(E,\epsilon)$ with respect to $\epsilon$ and \eqref{pac},  \eqref{nac},  we always have $\omega(E,0)=0$, then by Proposition \ref{equi} and Proposition \ref{uh}, we have  $$L(E)=0.$$  
For $|\lambda|>1$, by Proposition \ref{equi} and Proposition \ref{uh},   $L(E,\epsilon)$ can not be affine, which means $\omega(E,\epsilon)\neq \omega(E,-\epsilon)$ for any small $|\epsilon|$. By  the convexity  and \eqref{pac},  \eqref{nac}, we have $\omega(E,0)=0$ and $\omega(E,\epsilon)=-1$ for any $\epsilon<0$, which implies $$L(E)=\log |\lambda|.$$ 

To show $L(E)$ for all $E\in \mathbb{C}$, we need to add the calculation for $E\notin \Sigma$. Note that, by Proposition \ref{equi} and Proposition \ref{uh}, $\omega(E,\epsilon)$ is locally constant around $\epsilon=0$.  
By the convexity of $L(E,\epsilon)$, it is easy to see that if $\omega(E,0)=0$,
\begin{equation}\label{w=0}
	L(E)=\lim_{\epsilon\rightarrow +\infty} L(E,\epsilon)=\log \Big|\frac{E}{2}+\frac{\sqrt{E^{2}-4}}{2}\Big|.
\end{equation}
For $|\lambda|\leq 1$,  we have $\omega(E,0)=0$, then the result follows directly from \eqref{w=0}. 
For $|\lambda|>1$,  for better understanding the case,  denote 
\begin{equation}\label{IO}
	\begin{split}
			\mathcal{I}&=\bigg\{E\in \mathbb{C}: \log|\lambda| > \log \Big|\frac{E}{2}+\frac{\sqrt{E^{2}-4}}{2}\Big| \bigg\},\\
			\mathcal{O}&=\bigg\{E\in \mathbb{C}: \log|\lambda| < \log \Big|\frac{E}{2}+\frac{\sqrt{E^{2}-4}}{2}\Big| \bigg\}.
		\end{split}
\end{equation}
As shown in Figure \ref{LE}, if $\omega(E,0)=-1$, by \eqref{nac} we have $$L(E)=\log |\lambda|,$$ 
which corresponds to $E\in \mathcal{I}$, i.e. the interior of the ellipse;  if $\omega(E,0)=0$, then the result follows from \eqref{w=0} again, which corresponds to $E\in \mathcal{O}$, i.e. the outside of the ellipse. 

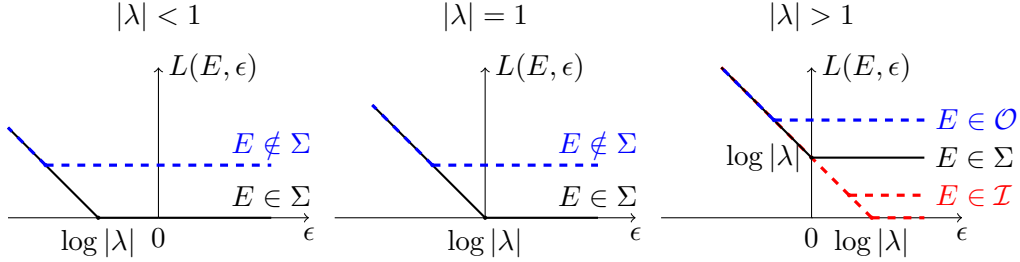
\begin{figure}[htbp]
	\centering
	\begin{tikzpicture}
		\draw[->] (-2,0) -- (2,0) node[below] {$\epsilon$};
		\draw[->] (0,0) -- (0,2) node[right] {$L(E,\epsilon)$};
		\draw[black, thick] (0,0) node {}-- (1.5,0) node[above] {$E\in \Sigma$};
		\draw[black, thick] (0,0) node {}-- (-0.8,0) node[right] {};
		\draw[black, thick] (-0.8,0) node {}-- (-2,1.2) node[right] {};
		\draw (0,0) -- (0,0) node[below] {$0$};
		\draw (-0.8,0) -- (-0.8,0) node [below]  {$\log|\lambda|$};
		\draw[blue, very thick,dashed] (-1.5,0.7) node {}-- (-2,1.2) node[right] {};
		\draw[blue, very thick,dashed] (-1.5,0.7) node {}-- (1.5,0.7) node[above] {$E\notin\Sigma$};
		\fill[black] (-0.8,0) ellipse (0.8pt and 0.8pt);
		\fill[blue] (-1.5,0.7) ellipse (0.8pt and 0.8pt);
		\draw (0,3) -- (0,3) node[below] {$|\lambda|<1$};
	\end{tikzpicture}
	\begin{tikzpicture}
		\draw[->] (-2,0) -- (2,0) node[below] {$\epsilon$};
		\draw[->] (0,0) -- (0,2) node[right] {$L(E,\epsilon)$};
		\draw[black, thick] (0,0) node {}-- (1.5,0) node[above] {$E\in \Sigma$};
		\draw[black, thick] (0,0) node {}-- (-1.5,1.5) node[right] {};
		\draw (0,0) -- (0,0) node[below] {$\log|\lambda|$};
		\draw[blue, very thick,dashed] (-0.7,0.7) node {}-- (-1.5,1.5) node[right] {};
		\draw[blue, very thick,dashed] (-0.7,0.7) node {}-- (1.5,0.7) node[above] {$E\notin \Sigma$};
		\fill[black] (0,0) ellipse (0.8pt and 0.8pt);
		\fill[blue] (-0.7,0.7) ellipse (0.8pt and 0.8pt);
		\draw (0,3) -- (0,3) node[below] {$|\lambda|=1$};
	\end{tikzpicture}
	\begin{tikzpicture}
		\draw[->] (-2,0) -- (2,0) node[below] {$\epsilon$};
		\draw[->] (0,0) -- (0,2) node[right] {$L(E,\epsilon)$};
		\draw[red, very thick,dashed] (0.5,0.3) node {}-- (-1.2,2) node[right] {};
		\draw[red, very thick,dashed] (0.5,0.3) node {}-- (1.5,0.3) node[right] {$E\in \mathcal{I}$};
		\draw[red, very thick,dashed] (0.8,0) node{}-- (0.5,0.3) node[right] {};
		\draw[red, very thick,dashed] (0.8,0) node {}-- (1.5,0) node[right] {};	
		\draw (0.8,0) -- (0.8,0) node[below] {$\log|\lambda|$};
		\draw (0,0.8) -- (0,0.8) node[left] {$\log|\lambda|$};
		\draw[black, thick] (0,0.8) node {}-- (1.5,0.8) node[right] {$E\in \Sigma$};
		\draw[black, thick] (0,0.8) node {}-- (-1.2,2) node[right] {};
		\draw (0,0) -- (0,0) node[below] {$0$};
		\draw[blue, very thick,dashed] (-0.5,1.3) node {}-- (-1.2,2) node[right] {};
		\draw[blue, very thick,dashed] (-0.5,1.3) node {}-- (1.5,1.3) node[right] {$E\in \mathcal{O}$};
		\fill[black] (0,0.8) ellipse (0.8pt and 0.8pt);
		\fill[red] (0.5,0.3) ellipse (0.8pt and 0.8pt);
		\fill[red] (0.8,0) ellipse (0.8pt and 0.8pt);
		\fill[blue] (-0.5,1.3) ellipse (0.8pt and 0.8pt);
		\draw (0,3) -- (0,3) node[below] {$|\lambda|>1$};
	\end{tikzpicture}
	\caption{Lyapunov exponent $L(E,\epsilon)$ for $E\in \mathbb{C}$.}
	\label{LE}
\end{figure}

\subsection{Proof of Theorem \ref{sarexample}:}
 We need distinguish two cases. 

{\bf Case 1: $|\lambda|\leq 1$.}  Note that $\log |\frac{E}{2}+\frac{\sqrt{E^{2}-4}}{2}|= 0$ if and only if $E\in [-2,2]$. If $E\in \Sigma$, by Theorem \ref{maintheorem3-1} we have $L(E)=0=\log |\frac{E}{2}+\frac{\sqrt{E^{2}-4}}{2}|$  which implies that $E\in [-2,2]$. On the contrary, if $E\notin \Sigma$, then $L(E)>0$ according to $(\alpha,S)\in \mathcal{UH}$, which implies that  $\log |\frac{E}{2}+\frac{\sqrt{E^{2}-4}}{2}|=L(E)>0$ by Theorem \ref{maintheorem3-1} and thus $E\notin [-2,2]$. 

{\bf Case 2: $|\lambda|>1$.}  Recall that $\mathcal{E}_{\lambda}=\{E:E=\lambda \mathrm{e}^{\mathrm{i}\theta}+\lambda^{-1}\mathrm{e}^{-\mathrm{i}\theta}, \theta\in [0,2\pi]\}$ and   $E\in \mathcal{E}_{\lambda}$ if and only if
$	\log|\lambda|= \log |\frac{E}{2}+\frac{\sqrt{E^{2}-4}}{2}|$. If $E\in \Sigma$, by Theorem \ref{maintheorem3-1} we deduce that $E\in \mathcal{E}_{\lambda}$.  If $E\notin \Sigma$, by Proposition \ref{equi} and Proposition \ref{uh}, we have $L(E)>0$ and $L(E,\epsilon)$ is affine with respect to $\epsilon$ around $\epsilon=0$. Hence we have either $E\in \mathcal{I}$ or $E\in \mathcal{O}$, see the definition in \eqref{IO} and the explanation in Figure \ref{LE}, and thus $E\notin \mathcal{E}_{\lambda}$.
\qed
\subsection{Proof of Proposition \ref{maintheorem2}:}

The proof is essentially contained in \cite{LZC}, we include the proof here just for completeness. 
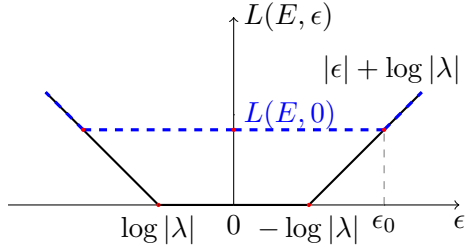
\begin{figure}[H]
	\centering
	\begin{tikzpicture}
		\draw[->] (-3,0) -- (3,0) node[below] {$\epsilon$};
		\draw[->] (0,0) -- (0,2.5) node[right] {$L(E,\epsilon)$};
		\draw[black, thick] (0,0) node {}-- (1,0) node[right] {};
		\draw[black, thick] (1,0) node {}-- (2.5,1.5) node[right] {};
		\fill[red] (1,0) ellipse (0.8pt and 0.8pt);
		\draw[black, thick] (0,0) node {}-- (-1,0) node[right] {};
		\draw[black, thick] (-1,0) node {}-- (-2.5,1.5) node[right] {};
		\fill[red] (-1,0) ellipse (0.8pt and 0.8pt);
		\draw (2.1,1.8) -- (2.1,1.8) node {$|\epsilon|+\log|\lambda|$};
		\draw (0,0) -- (0,0) node[below] {$0$};
		\draw (1,0) -- (1,0) node[below] {$-\log|\lambda|$};
		\draw (-1,0) -- (-1,0) node [below]  {$\log|\lambda|$};
		\draw[blue, very thick,dashed] (2,1) node {}-- (2.5,1.5) node[right] {};
		\draw[blue, very thick,dashed] (-2,1) node {}-- (-2.5,1.5) node[right] {};
		\draw[blue, very thick,dashed] (-2,1) node {}-- (0,1) node[right] {};
		\draw[blue, very thick,dashed] (0,1) node {}-- (2,1) node[right] {};
		\fill[red] (0,1) ellipse (0.8pt and 0.8pt);
		\draw (0,1.2) -- (0,1.2) node[blue, right] {$L(E,0)$};
		\draw[gray,dashed] (2,0) node {}-- (2,1) node[right] {};
		\draw[black] (2,0) -- (2,0) node [below] {$\epsilon_{0}$};
		\fill[red] (2,1) ellipse (0.8pt and 0.8pt);
		\fill[red] (-2,1) ellipse (0.8pt and 0.8pt);
	\end{tikzpicture}
	\caption{Lyapunov exponent $L(E,\epsilon)$ with $|\lambda|\in(0,1)$.}
	\label{acceration}
\end{figure}
Denote by $L(E,\epsilon)$ the Lyapunov exponent of $H_{v_{\epsilon},\alpha}$.  As shown in Figure \ref{acceration},  Avila \cite{Av0} proved that for any $E\in \mathbb{C}$, any $\epsilon\in\R$,
\begin{equation}\label{LE2-2}
	L(E,\epsilon)=\max \{\log |\lambda|+|\epsilon|,L(E,0)\}.
\end{equation}
In particular, 
then \begin{equation}\label{LE2}
	L(E,\epsilon)=\max \{\log |\lambda|+|\epsilon|,0\}, \quad  \forall \ E\in  \Sigma_{2\lambda \cos,\alpha}.
\end{equation}

Suppose that $E\in  \Sigma_{2\lambda \cos,\alpha}$, since $|\lambda|<1$ and $|\epsilon|<-\log |\lambda|$, it follows from \eqref{LE2} that $L(E,\epsilon)=0$ and $L(E,\epsilon+t)=0$ when $|t|\leq \log |\lambda|-|\epsilon|$, thus $E\in \Sigma_{v_\epsilon,\alpha}$ according to Proposition \ref{equi} and Proposition \ref{uh}.

Suppose that $E\notin  \Sigma_{2\lambda \cos,\alpha}$, by Proposition \ref{equi} we have  $L(E,0)>0$ and  $(\alpha, S_{E,v_{0}}) \in \mathcal{UH}$. Then it follows from \eqref{LE2-2} that $L(E,\epsilon_{0})=L(E,0)$ for $|\epsilon_{0}|\leq L(E,0)-\log|\lambda|$. Since $|\epsilon|<-\log |\lambda|$ by assumption, we have $L(E,\epsilon)=L(E,0)>0$ and $L(E,\epsilon+t)$ is affine for $|t|< L(E,0)$, it follows from Proposition \ref{equi} and Proposition \ref{uh} again that $E\notin \Sigma_{v_\epsilon,\alpha}$. We thus finish the whole proof.
\qed

\appendix

\section{Proof of Lemma \ref{hy}}
	We are going to use induction to show 
	\begin{equation*}
		{\rm e}^{-Y_{j-1}(x+\alpha)} A{\rm e}^{F_{j-1}^{\rm nre}(x) + F_{j-1}^{\rm re}(x)} {\rm e}^{Y_{j}(x)} =A{\rm e}^{F_{j}^{\rm nre}(x)+F_{j}^{\rm re}(x)}, \ j\geq 1,
	\end{equation*}
	with estimates
	\begin{equation}\label{e}
		\|Y_{j-1}\|_{h}\leq \varepsilon_{j-1}^{\frac{3}{4}},\  \|F_{j}^{\rm re}- F_{j-1}^{\rm re}\|_{h}\leq \varepsilon^{\frac{1}{3}}\varepsilon_{j-1}, \ \|F_{j}^{\rm nre}\|_{h}\leq \varepsilon_{j},
	\end{equation}
	where the sequences are defined as
	\begin{equation*}
		\varepsilon_{j}=\varepsilon_{j-1}^{\frac{3}{2}}, \ \ \varepsilon_{0}=\varepsilon, \ \ F_{0}=F, \ \ F_{0}^{\rm nre}= \mathbb{P}_{\rm nre} F, \ \  F_{0}^{\rm re}= \mathbb{P}_{\rm re} F.
	\end{equation*} 
	
	Suppose that for $j=n$, we obtain $(\alpha, A{\rm e}^{F^{\rm re}_{n}+ F^{\rm nre}_{n}})$ and \eqref{e} holds. For any $Y\in \mathcal{B}_{h,r}^{\rm nre}(\sigma)$, we define
	\begin{equation*}
		\begin{split}
			&\tilde{J}(Y):= \log {\rm e}^{-Y}{\rm e}^{F_{n}^{\rm re}} +Y-F_{n}^{\rm re}, \\
			&\tilde{K}(Y):= \log {\rm e}^{-Y}{\rm e}^{F_{n}^{\rm re}} + \log {\rm e}^{F_{n}^{\rm re}}{\rm e}^{Y} - 2F_{n}^{\rm re},
		\end{split}
	\end{equation*} 
Let $J(Y)$ (resp. $K(Y)$) be the linear part of $\tilde{J}(Y)$ (resp. $\tilde{K}(Y)$) with respect to $Y$, 
	\begin{equation*}
		J(\cdot), K(\cdot) : \quad \mathcal{B}_{h,r} \rightarrow \mathcal{B}_{h,r}.
	\end{equation*}
	Define the sequences for $j\in \mathbb{N}$ as
	\begin{equation*}
		Q_{j+1}=(-1)^{j}J(Q_{j}), \ R_{j+1}=(-1)^{j}J(R_{j}), \ Q_{0}= K(Y), \ R_{0}=F_{n}^{\rm nre}.
	\end{equation*}
	Let us consider the linear operator $I_{A}: \mathcal{B}_{h,r}^{\rm nre}(\sigma) \rightarrow \mathcal{B}_{h,r}^{\rm nre}(\sigma)$ given by
	\begin{equation*}
		I_{A}Y =L_{A}Y-\sum_{j=0}^{2^n-1}Q_{j}(Y) =  A^{-1}Y(x+\alpha)A-Y(x)- \sum_{j=0}^{2^n-1}Q_{j}(Y(x)).
	\end{equation*}
	Since $\|F^{\rm re}_{n}\|_{h}\leq 2\varepsilon$, we have $\|I_{A}Y\|\geq \frac{3}{4}\varepsilon^{\frac{1}{2}} \|Y\|_{h}$, and thus $\|I^{-1}_{A}\|$ is bounded by $\frac{4}{3} \varepsilon^{-\frac{1}{2}}$. There exists $Y_{n}$ such that $I_{A}Y_{n} = \mathbb{P}_{\rm nre} \sum_{j=0}^{2^n-1}R_{j}$, i.e.,
	\begin{equation*}
		A^{-1}Y_{n}(x+\alpha)A-Y_{n}(x) -\sum_{j=0}^{2^n-1}Q_{j}(Y_{n})= \mathbb{P}_{\rm nre} \sum_{j=0}^{2^n-1}R_{j}.
	\end{equation*}
	Moreover, $\|Y_{n}\|_{h}\leq \frac{4}{3} \varepsilon^{-\frac{1}{2}} (\varepsilon_{n}+4\varepsilon \varepsilon_{n}) \leq 2 \varepsilon^{-\frac{1}{2}} \varepsilon_{n}$. Thus, 
	\begin{equation*}
		\begin{split}
			{\rm e}^{F_{n+1}^{\rm nre}(x)+F_{n+1}^{\rm re}(x)}&= {\rm e}^{-A^{-1}Y_{n}(x+\alpha)A}{\rm e}^{F_{n}^{\rm nre}(x) + F_{n}^{\rm re}(x)} {\rm e}^{Y_{n}(x)} \\
			& = {\rm e}^{-Y_{n}(x)-\mathbb{P}_{\rm nre} \sum_{j=0}^{2^n-1}R_{j}-\sum_{j=0}^{2^n-1}Q_{j}}{\rm e}^{F_{n}^{\rm nre}(x) + F_{n}^{\rm re}(x)} {\rm e}^{Y_{n}(x)},
		\end{split}
	\end{equation*}
	Let us recall the Baker-Campbell-Hausdorff Formula,
	\begin{equation*}
		\log ({\rm e}^{X}{\rm e}^{W}{\rm e}^{Z})  = X+W+Z+\frac{1}{2}[X,W]
		+\frac{1}{2}[W,Z] +\frac{1}{2}[X,Z]	+O^{3}(X,Y,Z),
	\end{equation*}
	where $O^{3}(X,Y,Z)$ stands for the sum of terms whose Lie brackets involving three elements of $X,W,Z$. By the construction of $R_{j}, Q_{j}$ and B-C-H formula,  
	\begin{equation*}
		\begin{split}
			&F_{n+1}^{\rm re}(x)= F_{n}^{\rm re}(x) + \mathbb{P}_{\rm re} \bigg\{ -\frac{1}{2}[F_{n}^{\rm nre}, F_{n}^{\rm re}]+ [F^{\rm re}_{n},Y_{n}]+ \cdots\bigg\},\\
			&F_{n+1}^{\rm nre}(x)= \mathbb{P}_{\rm nre} \bigg\{ -\frac{1}{4}[Y_{n},[F^{\rm re}, Y_{n}]] -\frac{1}{2}[R_{2^n-1}, F^{\rm re}] -\frac{1}{2}[Q_{2^n-1}, F^{\rm re}]+ \cdots\bigg\}.
		\end{split}
	\end{equation*}
	Since $\|Q_{j}\|_{h} \leq 2\varepsilon \|Q_{j-1}\|_{h}$ and $\|R_{j}\|_{h} \leq 2\varepsilon \|R_{j-1}\|_{h}$, we get that 
	\begin{equation*}
		\begin{split}
			&\|Q_{2^n-1}\|_{h}\leq (2\varepsilon)^{2^n-1} \|Q_{0}\|_{h}\leq (2\varepsilon)^{2^n} \|Y_{n}\|_{h},\\
			&\|R_{2^n-1}\|_{h}\leq (2\varepsilon)^{2^n-1} \|F^{\rm nre}_{0}\|_{h} \leq (2\varepsilon)^{2^n}.
		\end{split}
	\end{equation*}
Thus, we deduce that 
	\begin{equation*}
		\|F^{\rm re}_{n+1}- F^{\rm re}_{n}\|_{h} \leq \varepsilon^{\frac{1}{3}}\varepsilon_{n}, \quad \|F^{\rm nre}_{n+1}\|_{h} \leq \varepsilon_{n}^{\frac{3}{2}}=\varepsilon_{n+1}.
	\end{equation*}
	
	Now we let $Y=\log (\prod_{n=0}^{\infty} {\rm e}^{Y_{n}})$ and $F^{\rm re} =\lim_{n\rightarrow \infty} F^{\rm re}_{n}$.  Since $\mathcal{B}_{h,r}^{\rm re}(\sigma)$ and $\mathcal{B}_{h,r}^{\rm nre}(\sigma)$ are closed subspace in $\mathcal{B}_{h,r}$, it follows from Proposition \ref{banach} that $Y\in \mathcal{B}_{h,r}$ and $F^{\rm re}\in \mathcal{B}_{h,r}^{\rm re}(\sigma)$. By direct calculation, we have $\|Y\|_{h}\leq \varepsilon^{\frac{1}{2}}$ and $\|F^{\rm re} - \mathbb{P}_\mathrm{re} F\|_{h}\leq 2\varepsilon^{\frac{4}{3}}$. \qed

\section*{Acknowledgements} 
This work was  partially supported by National Key R\&D Program of China (2020 YFA0713300) and Nankai Zhide Foundation.  J. You was also partially supported by NSFC grant (11871286). Q. Zhou was supported by NSFC grant (12071232), the Science Fund for Distinguished Young Scholars of Tianjin (No. 19JCJQJC61300).

\end{document}